\documentclass[10pt,twocolumn,twoside]{IEEEtran}

\usepackage{amssymb,amsfonts,amsmath,amsthm}
\usepackage{mathtools}
\usepackage{eucal,bm}
\usepackage{gensymb,subfigure}
\usepackage{multirow}
\usepackage{booktabs}
\usepackage{times,color,url}
\usepackage{comment}
\usepackage{cite}
\usepackage{graphicx}
\usepackage[all]{xy}
\usepackage{algorithm}
\usepackage{algpseudocode}
\usepackage{subfigure}
\usepackage{bm}
\usepackage{color}
\usepackage{wrapfig}
\usepackage{epsf}
\usepackage{times,mathptm}
\usepackage{url}
\usepackage{nomencl}
\newtheorem{theorem}{Theorem}

\newtheorem{lemma}{Lemma}
\newcommand{\squeezeup}{\vspace{-2.5mm}}
\begin{document}
\title{Structure Learning and Statistical Estimation in Distribution Networks - Part II}
\author{\authorblockN{Deepjyoti~Deka*, Scott~Backhaus\dag, and Michael~Chertkov\ddag\\}
\authorblockA{*Corresponding Author. Electrical \& Computer Engineering, University of Texas at Austin\\
\dag MPA Division, Los Alamos National Lab\\ \ddag Theory Division and the Center for Nonlinear Systems, Los Alamos National Lab}
\thanks{D. Deka is with the Department of Electrical and Computer Engineering, The University of Texas at Austin, Austin, TX 78712. Email: deepjyotideka@utexas.edu}
\thanks{S. Backhaus is with the MPA Division of LANL, Los Alamos, NM 87544. Email: backhaus@lanl.gov}
\thanks{M. Chertkov is with the Theory Division and the Center for Nonlinear Systems of LANL, Los Alamos, NM 87544. Email: chertkov@lanl.gov}}
\maketitle

\begin{abstract}
Part I \cite{distgridpart1} of this paper discusses the problem of learning the operational structure of the grid from nodal voltage measurements. In this work (Part II), the learning of the operational radial structure is coupled with the problem of estimating nodal consumption statistics and inferring the line parameters in the grid. Based on a Linear-Coupled (LC) approximation of AC power flows equations, polynomial time algorithms are designed to complete these tasks using the available nodal complex voltage measurements. Then the structure learning algorithm is extended to cases with missing data, where available observations are limited to a fraction of the grid nodes. The efficacy of the presented algorithms are demonstrated through simulations on several distribution test cases.
\end{abstract}

\begin{IEEEkeywords}
Power Distribution Networks, Power Flows, Struture/graph Learning, Load estimation, Parameter estimation, Voltage measurements, Transmission Lines, Missing data.
\end{IEEEkeywords}
\section{Introduction}
\label{sec:intro}
The present power grid is separated into different tiers for optimizing its operations and control, namely the high voltage transmission system and the medium and low voltage distribution system. The distinction between these systems extends to their operational structure: the transmission system is a loopy graph while the distribution system operates as a radial network (set of trees). The larger volume of power transferred and higher magnitudes of resident voltages in the transmission network as compared to the distribution network have led grid security and reliability studies to focus primarily on the transmission side. Traditionally, the distribution grid has thus suffered from low placement of measurement devices leading to negligible real-time observation and control efforts \cite{hoffman2006practical}.

In Part I \cite{distgridpart1} of this paper, we study the design of low-complexity algorithms for learning the operational radial structure of the distribution grid despite available metering limited to nodal voltages. In this work, we extend the study to the problem of estimating other features of the distribution grid together with learning the operational structure. Specifically, we utilize available node complex voltages to learn the statistics of load profiles at the grid nodes and to estimate the complex-valued impedance parameters of the operational distribution lines. It is worth noting that line/edge based metering (line flow and breaker status measurements) are considered unavailable as they are seldom observed in real time in today's grids. Next, we extend the problem of learning the grid structure introduced in Part I to the case with partial observability, where voltage measurements pertaining to a subset of the nodes are not observed. In essence, the results from this work can aid several areas that have gained prominence with the expansion of smart grid. These include failure identification \cite{sharon2012topology}, grid reconfiguration \cite{baran1989network}, power flow optimization and generation scheduling \cite{hoffman2006practical,lopes2011integration,baran1989network,turitsyn2011options}, as well as privacy preserving grid operation \cite{liu2012cyber}. Furthermore, learning under partial observability enables the quantification of measurement security necessary to prevent adversarial learning aimed at hidden topological attacks \cite{kim2013topology,deka2014attacking}.

`Graph Learning' or `Graphical Model Learning' \cite{wainwright2008graphical} is a broad area of work that has been considered in different domains. In general graphs, maximum-likelihood has been employed for learning graph structures \cite{ravikumar2010high, anandkumar2011high, netrapalli2010greedy} through convex optimization as well as greedy techniques. In a learning study specific to general power grids \cite{kekatos2013grid} presents a maximum likelihood structure estimator (MLE) based on electricity prices. For radial distribution grids, the authors of \cite{bolognani2013identification} discuss structure learning through construction of a spanning tree based on the inverse covariance matrix (or concentration matrix) of voltage measurements, while \cite{sharon2012topology} studies topology identification with Gaussian loads through a maximum likelihood scheme.

In Part I \cite{distgridpart1}, an approach that uses provable trends in second moments of nodal voltage magnitudes to learn the grid structure was presented. Our algorithm design in part I assumes that all nodal loads are, in expectation, consumers of active and reactive power which is realistic for most, if not all, current distribution grids. Here in part II, we use a modified but not conflictive assumption of independence of fluctuations in active and reactive loads at different nodes.
As shown below, under this assumption one is not only able to reconstruct the grid structure but also able to infer either the statistics of active and reactive loads at every node or the values of impedance parameters at every operational line. Then, we show how to extend our structure learning algorithm to cases with missing data, where observations from a subset of nodes are not available to the observer. Similarly to Part I, the algorithms in here (Part II) are independent of the exact probability distribution of load profiles as well as variations in values of line parameters and are thus applicable to a wide range of operational conditions.

The rest of this manuscript (part II) is organized as follows. Section \ref{sec:tech_intro} contains a brief review of the radial structure of the grid, approximations of power flows and sets formulation of problems considered. Section \ref{sec:trends} contains proofs of our main results on second moments of voltage measurements in radial grids. Section \ref{sec:algo1} describes the algorithm design to learn the operational structure and estimate the statistics of load power profiles in the grid. An extension is also discussed for the problem of structure learning coupled with estimation of line impedances (instead of injection statistics).  In Section \ref{sec:missing} we present Algorithm $2$ that learns the operational radial structure in the presence of missing observations. Simulations results for our Algorithms on test radial distribution cases are presented in Section \ref{sec:experiments}. Finally, conclusions are discussed in Section \ref{sec:conclusions}.

\section{Technical Preliminaries}
\label{sec:tech_intro}
This Section provides a brief description of the operational structure of the distribution grid, and introduces the learning problems considered in Part II. We then have a brief reminder about the Linear Coupled Power Flow (LC-PF) model (already introduced and discussed in Part I) that we rely on for analysis in later Sections.

\textbf{Structure of Radial Distribution Network:} A distribution grid is represented by a graph ${\cal G}=({\cal V},{\cal E})$, where ${\cal V}$ (of size $N+K$) is the set of nodes/buses and ${\cal E}$ is the set of undirected edges/transmission lines. 
The complete layout of $\cal G$ is loopy, but its operational layout (denoted by $\cal F$) derived by excluding open/non-operational lines is a union of $K$ non-intersecting trees. Each grid tree ${\cal T}_k$ in $\cal F$ comprises of a single substation feeding electricity into load nodes lined along the `radial' tree. Thus, ${\cal F}$ is a $K$ \textit{\textbf{`base-constrained spanning forest'}} with $N$ non-substation nodes. See Fig. $1$ in Part I \cite{distgridpart1} for an illustrative example. The set of operational edges that contribute to the structure of the forest $\cal F$ is denoted by ${\cal E}^{\cal F}$ where ${\cal E}^{\cal F}\subset {\cal E}$. We follow the same notation as Part I and described in Table I of \cite{distgridpart1}.

\textbf{Summary of Learning Problems:} The majority of distribution grids operational today are handicapped by limited real time metering for breaker statuses and power flows \cite{hoffman2006practical}, as well as infrequent updating of model parameters. The grid operator (utility company) or an external observer/adversary in such a scenario is concerned with the following three tasks:
\begin{itemize}
\item[(1)] To learn the current configuration of switches that determine the `base-constrained spanning forest'.
\item [(2)] To learn the statistics of the power consumption\footnote{We use the term `power injection', `power consumption' and `load' interchangeably to denote power profile at each interior (non-substation) node of the distribution system.} profiles at the nodes.
\item [(3)] To learn the values of resistances and reactances of each operational  line of the distribution system.
\end{itemize}

For all these tasks, the utility or observer relies on available nodal complex voltage (magnitude and phase) readings. Task (1) is coupled with either Task (2) or Task (3) and considered first in the situation of full observability, when complex voltage (magnitude and phase) samples are available at all the nodes of the system. In fact, we show that voltage magnitude samples are sufficient to learn the grid structure (Task (1)), additional voltage phasor measurements are needed for the inference problems in Tasks (2) and (3). However, we also discuss Task (1) independently in the situation where several nodes do not offer any voltage readings. The problem formulations considered in Part I previously and in Part II are summarized in Table \ref{table1}.

\begin{table*}[ht]
\caption{Summary of Learning Problems/Statements}
\begin{center}
\begin{tabular}{|p{2.5cm}|p{4cm}|p{3.8cm}|p{3.2cm}|p{3cm}|}
\hline
Observations available & Prior Information & Assumptions & Features estimated & Results used \\\hline
Voltage magnitudes of all nodes & True second moment of nodal power injections, resistance and reactance of edges & Non-negative second moments of nodal power injections & Operational network structure& Algorithm $1$ in Part I \cite{distgridpart1}\\\hline
Voltage magnitudes of all nodes & None & Uncorrelated nodal power injections & Operational network structure (Task (1)) & Theorem \ref{Theorem1_LC}, Theorem \ref{Theorem4}, Algorithm $1$\\\hline
Voltage magnitudes and phasors of all nodes & Resistance and reactance of edges & Uncorrelated nodal power injections & Mean and variance of nodal power injections (Task (2))& Lemma \ref{LemmadiffsqLC}, Algorithm $1$\\\hline
Voltage magnitudes and phasors of all nodes & True variance of nodal power injection & Uncorrelated nodal power injections & Resistance and reactance of operational lines (Task (3))& Lemma \ref{LemmadiffsqLC}, Algorithm $1$\\\hline
Voltage magnitudes of subset of nodes & True variance of nodal power injections, resistance and reactance of edges & Uncorrelated nodal power injections, Missing nodes separated by three or more hops & Operational network structure& Theorem \ref{Theorem1_LC}, Lemma \ref{LemmadiffsqLC}, Algorithm $2$\\\hline
\end{tabular}
\end{center}
\label{table1}
\squeezeup
\end{table*}

The physics of Power Flows (PFs) in $\cal F$ forms the background for the learning/reconstruction problems sketched here. Variety of PF models/approximations  were discussed in details in Appendix $1$ and Section IIIA-C of Part I \cite{distgridpart1}. Let us briefly recap essential features of the Linear-Coupled Power Flow (LC-PF) model essential for analysis presented in the following Sections, also extending it with some new notations.

\textbf{Linear Coupled Power Flow (LC-PF):} Let $r^{\cal F}$ and $x^{\cal F}$ denote the diagonal matrices representing, respectively, line resistances and reactances for operational edges in forest ${\cal F}$. 
Let $N \times 1$ real valued vectors $p, q, \varepsilon$ and $\theta$ denote the active power injections, reactive power injections, voltage magnitude deviations and voltage phasors at the non-substation nodes, respectively. The LC-PF model is given by the matrix Eqs~(5,6) of Part I, where, $H_g$ and $H_{\beta}$ are edge-weighted reduced graph Laplacian matrices (after removing sub-station/slack buses) for forest ${\cal F}$ with edge weights given by the edge conductances and susceptances respectively. $M$ is the reduced directed incidence matrix with each row corresponding to a directed edge $(ab)$ in ${\cal E}^{\cal F}$. In fact, $M$ is block diagonal with $M=\mbox{diag}(M_1,M_2,\cdots, M_K)$, where each block ($M_i$) corresponds to a tree ${\cal T}_i$ in $\cal F$. Assuming that $p$ and $q$ in Eqs.~$(5,6)$ of Part I
are fluctuating, we derive the following relations involving the means $\mu_{x}$, and covariance matrices $\Omega_{xy}\doteq\mathbb{E}[(x-\mu_x)(y-\mu_y)^T]$ for variables $x$ and $y$.
\begin{align}
\mu_{\theta} &= H^{-1}_{1/x}\mu_p - H^{-1}_{1/r}\mu_q,~~\mu_\varepsilon = H^{-1}_{1/r}\mu_p + H^{-1}_{1/x}\mu_q\label{means}\\
\Omega_{\theta} &= H^{-1}_{1/x}\Omega_{p}H^{-1}_{1/x} + H^{-1}_{1/r}\Omega_qH^{-1}_{1/r}-H^{-1}_{1/x}\Omega_{pq}H^{-1}_{1/r} \nonumber\\
&~-H^{-1}_{1/r}\Omega_{qp}H^{-1}_{1/x}\label{angcovar1}\\
\Omega_{\varepsilon} &= H^{-1}_{1/r}\Omega_{p}H^{-1}_{1/r} + H^{-1}_{1/x}\Omega_qH^{-1}_{1/x}+H^{-1}_{1/r}\Omega_{pq}H^{-1}_{1/x}\nonumber\\
&~+H^{-1}_{1/x}\Omega_{qp}H^{-1}_{1/r}\label{volcovar1}\\
\Omega_{\theta\varepsilon} &= H^{-1}_{1/x}\Omega_{p}H^{-1}_{1/r} - H^{-1}_{1/r}\Omega_qH^{-1}_{1/x} + H^{-1}_{1/x}\Omega_{pq}H^{-1}_{1/x}\nonumber\\
&~- H^{-1}_{1/r}\Omega_{qp}H^{-1}_{1/r}\label{volangcovar1}\\
\Omega_{\varepsilon\theta} &= H^{-1}_{1/r}\Omega_{p}H^{-1}_{1/x} - H^{-1}_{1/x}\Omega_qH^{-1}_{1/r} + H^{-1}_{1/x}\Omega_{qp}H^{-1}_{1/x} \nonumber\\
&~- H^{-1}_{1/r}\Omega_{pq}H^{-1}_{1/r}\label{angvolcovar1}
\end{align}

It is worth mentioning that inclusion of both line resistances and reactances in the LC-PF model distinguishes it from the DC power flow models \cite{abur2004power} that has limited applicability in distribution grids. In the next Section, we derive key results relating second moments in phase angles and voltage magnitudes in the LC-PF for a radial distribution grid. Versions of all subsequent results can be generated for DC power flow models by simply ignoring line resistances or reactances as demonstrated in Part I.

\section{Second Moments of Voltages in Radial Grids}
\label{sec:trends}
Consider a tree ${\cal T}_k$ with reduced incidence matrix $M_k$. Let ${\cal E}_a^{{\cal T}_k}$ denote the unique path from node $a$ to the slack bus of the tree ${\cal T}_k$, where path between two nodes refers to the unique set of edges connecting them. As shown in Part I \cite{distgridpart1}, in a radial distribution gird, $H_{1/r}^{-1}$ has the following structure,
\begin{align}
 H_{1/r}^{-1}(a,b)&= \sum_{f} M^{-1}(a,f){r^{{\cal F}}}(f,f)M^{-1}(b,f)\nonumber\\
 &= \begin{cases} \sum_{(cd) \in {\cal E}_a^{{\cal T}_k}\bigcap {\cal E}_b^{{\cal T}_k}} r_{cd} \text{~~if nodes~} a,b \in {\cal T}_k\\
 0 ~~ \text{otherwise,} \end{cases}\label{Hrxinv}
\end{align}

Let $D^{{\cal T}_k}_a$ denote the set of descendants of node $a$ within the tree ${\cal T}_k$ where $b$ is called a descendent of $a$, if $a$ lies on the (unique) path from $b$ to the slack bus of ${\cal T}_k$. We include $a$ itself in the set of its descendants. Similarly, we call $b$ the parent of $a$ within ${\cal T}_k$ if $a$ is an immediate descendant of $b$ as illustrated in Fig \ref{fig:descendant}.
\begin{figure}[!bt]
\centering
\subfigure[]{\includegraphics[width=0.20\textwidth]{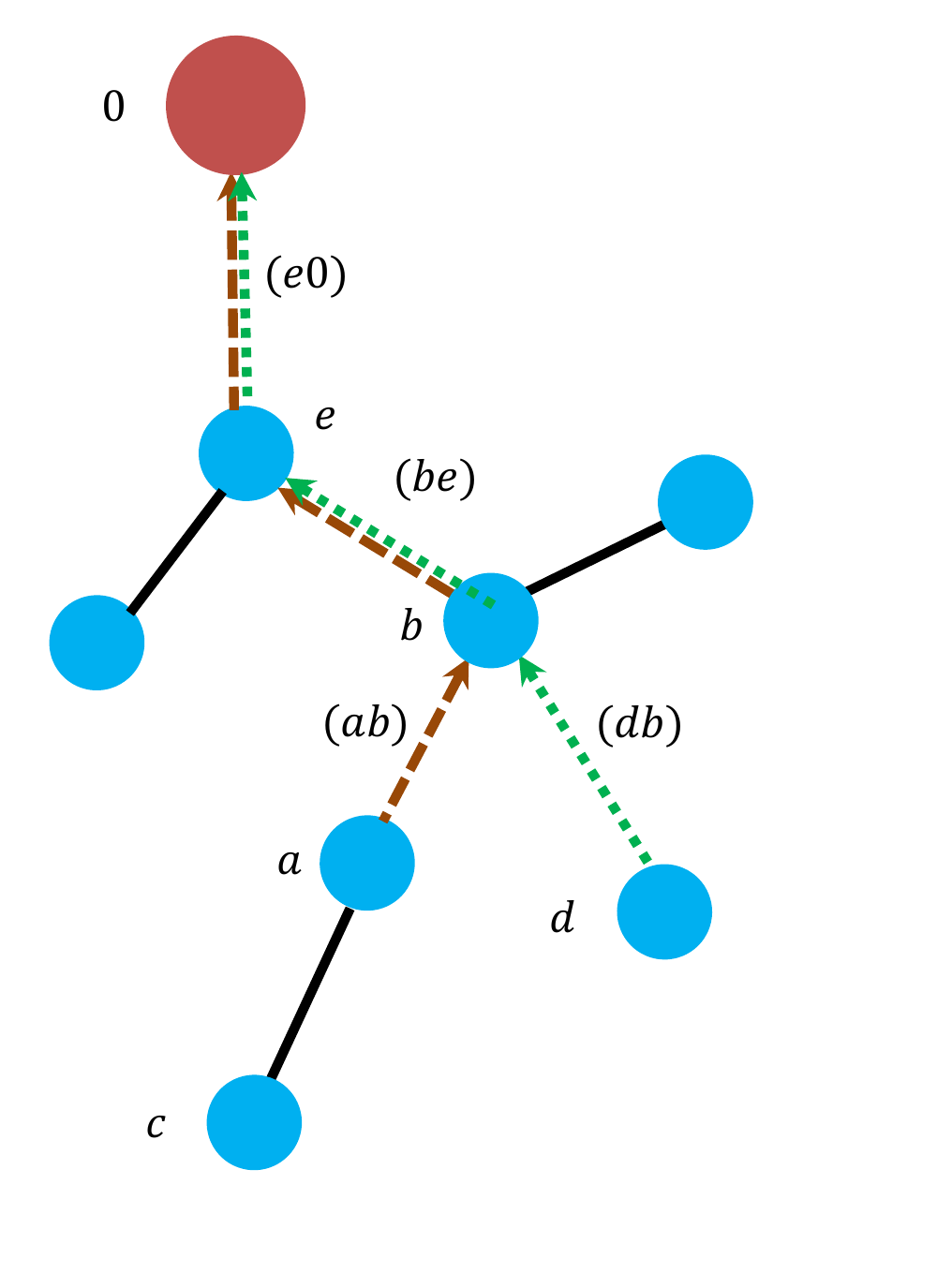}\label{fig:picHinv1}}\hspace{.6cm}
\subfigure[]{\includegraphics[width=0.20\textwidth]{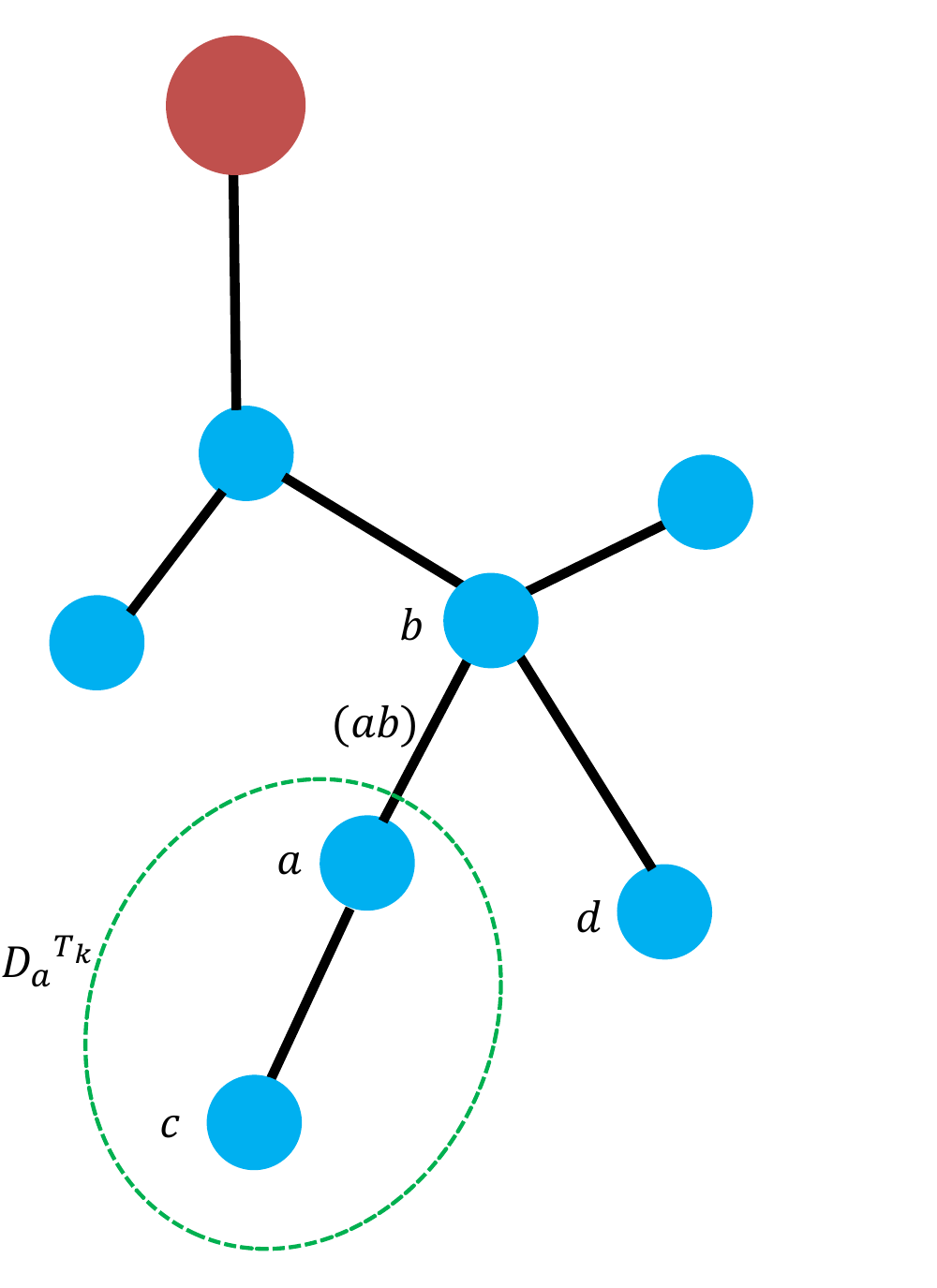}\label{fig:descendant}}
\squeezeup
\caption{Schematic layout of a distribution grid tree ${\cal T}_k$. The sub-station node represented by large red node is the slack bus. (a) Dotted lines represent the paths from nodes $a$ and $d$ to the slack bus. Here, $H_{1/r}^{-1}(a,d) = r_{be}+ r_{e0}$. (b) Here, nodes $a$ and $c$ are descendants of node $a$.
\label{fig:picHinv}}
\end{figure}

The following statement holds (see Lemma $1$ in \cite{distgridpart1} for detailed proof).
\begin{lemma}\label{Lemmadiff}
For two nodes, $a$ and its parent $b$, in tree ${\cal T}_k$
\begin{align} 
{\huge H}_{1/r}^{-1}(a,c)-{\huge H}_{1/r}^{-1}(b,c) &&=\begin{cases}r_{ab} & \quad\text{if node $c \in D^{{\cal T}_k}_a$}\\
0 & \quad\text{otherwise,} \end{cases} \label{Hdiff}
\end{align}
\end{lemma}

Before the discussion of our results on trends in voltage covariances, we make the following assumption on the covariances of load consumption profiles.

\textbf{Assumption $1$:} Powers at different nodes are not correlated, while active and reactive powers at the same node are positively correlated. Thus, $\forall a,b \in \{1,...,N\}$
\begin{align}
\Omega_{qp}(a,a) > 0,~\Omega_p(a,b) = \Omega_q(a,b)= \Omega_{qp}(a,b) = 0 \nonumber
\end{align}

Few remarks are in order. First, the assumption of independence of fluctuations is realistic in general, reflecting diversity of individual consumer behavior on relatively short time scales. Second, unless consumer-level control of reactive power is implemented \cite{KostyaMishaPetrScott2} is implemented, fluctuations in active and reactive consumption/generation at the same node will have a strong tendency to align, giving positive correlation. Since, Assumption $1$ pertains to covariances (`centered' second moments), it does not run counter to the assumption in Part I, where `non-centered' second moments of power injections are considered to be positive. In fact, nodal loads (consumers of active and reactive power) satisfy both the assumptions given in Part I and Part II. Note that Assumption $1$ does not restrict individual nodal loads to follow any specific distribution.

The following result states that covariances of voltage magnitude deviations increase as we move farther away from the root of any tree in the grid.

\begin{theorem} \label{Theorem1_LC}
If node $a \neq b$ is a descendant of node $b$ on tree ${\cal T}_k$ in forest $\cal F$, then $\Omega_{\varepsilon}(a,a) > \Omega_{\varepsilon}(b,b)$.
\end{theorem}
\begin{proof}
$\Omega_{\varepsilon}$ is given by Eq.~(\ref{volcovar1}) with four non-negative terms on the right side. Let the first term $H^{-1}_{1/r}\Omega_pH^{-1}_{1/r}$ be denoted by $\Omega^1_{\varepsilon}$. For one-hop neighbors, node $a$ and its parent $b$, we use Lemma \ref{Lemmadiff} to get
\begin{align}
\Omega^1_{\varepsilon}(a,a) - \Omega^1_{\varepsilon}(a,b) &= \sum_{c \in D^{{\cal T}_k}_a}H_{1/r}^{-1}(a,c)\Omega_p(c,c)r_{ab} > 0 \label{uselater1}\\
\Omega^1_{\varepsilon}(a,b) - \Omega^1_{\varepsilon}(b,b) &= \sum_{cc \in D^{{\cal T}_k}_a}H_{1/r}^{-1}(b,c)\Omega_p(c,c)r_{ab}> 0
\end{align}
Combining the inequalities, we get $\Omega^1_{\varepsilon}(a,a) > \Omega^1_{\varepsilon}(b,b)$. Extending the same analysis to the remaining three terms in Eq.~(\ref{volcovar1}) and then moving from one-hop neighbors to descendants proves the theorem.
\end{proof}

Next, we focus on the term $\mathbb{E}[(\varepsilon_a - \mu_{\varepsilon_a})-(\varepsilon_b-\mu_{\varepsilon_b})]^2 $, which is the expected value of the squared centered difference between two node voltage deviations ($\varepsilon$). For any two nodes $a$ and $b$ that lie on tree ${\cal T}_k$, we have
\begin{align}
\mathbb{E}[(\varepsilon_a - \mu_{\varepsilon_a})-(\varepsilon_b-\mu_{\varepsilon_b})]^2 &= \Omega_{\varepsilon}(a,a) - \Omega_{\varepsilon}(a,b) \nonumber\\&~+ \Omega_{\varepsilon}(b,b)- \Omega_{\varepsilon}(b,a) \nonumber
\end{align}
where $\Omega_{\varepsilon}$ is composed of four terms as given by Eq.~(\ref{volcovar1}). Using Eq.~\ref{uselater1} for each of the four terms within $\Omega_{\varepsilon}$ and adding them, we derive
\begin{align}
&\mathbb{E}[(\varepsilon_a - \mu_{\varepsilon_a})-(\varepsilon_b-\mu_{\varepsilon_b})]^2 =
\smashoperator[lr]{\sum_{c \in {\cal T}_k}}(H^{-1}_{1/r}(a,c)- H^{-1}_{1/r}(b,c))^2\Omega_p(c,c)\nonumber\\ &+(H^{-1}_{1/x}(a,c)- H^{-1}_{1/x}(b,c))^2 \Omega_q(c,c)+2\left(H^{-1}_{1/r}(a,c)- H^{-1}_{1/r}(b,c)\right)\nonumber\\
&\left(H^{-1}_{1/x}(a,c)- H^{-1}_{1/x}(b,c)\right)\Omega_{pq}(c,c) \label{usediff_1}
\end{align}

For the special case where node $b$ is the parent of node $a$, using Lemma \ref{Lemmadiff} in Eq.~(\ref{usediff_1}), we obtain
\begin{lemma} \label{LemmadiffsqLC}
If $b$ is $a$'s parent in tree ${\cal T}_k$,
\begin{align}
&\mathbb{E}[(\varepsilon_a - \mu_{\varepsilon_a})-(\varepsilon_b-\mu_{\varepsilon_b})]^2 = \sum_{c \in D^{{\cal T}_k}_a} r_{ab}^2\Omega_p(c,c) + x_{ab}^2\Omega_q(c,c)\nonumber\\
&~~~~~~~~~~~~~~~~~~~~~~~~~~~~~~~~~+ 2r_{ab}x_{ab}\Omega_{pq}(c,c)\label{test1}\\
& \mathbb{E}[(\theta_a - \mu_{\theta_a})-(\theta_b-\mu_{\theta_b})]^2 = \sum_{c \in D^{{\cal T}_k}_a} x_{ab}^2\Omega_p(c,c) + r_{ab}^2\Omega_q(c,c)\nonumber\\
&~~~~~~~~~~~~~~~~~~~~~~~~~~~~~~~~~- 2r_{ab}x_{ab}\Omega_{pq}(c,c)\label{test2}\\
&\mathbb{E}[(\varepsilon_a-\mu_{\varepsilon_a}-\varepsilon_b+\mu_{\varepsilon_b})(\theta_a - \mu_{\theta_a}-\theta_b+\mu_{\theta_b})] =\nonumber\\
&\smashoperator[lr]{\sum_{c \in D^{{\cal T}_k}_a}}r_{ab}x_{ab}(\Omega_p(c,c)-\Omega_q(c,c)) + (x_{ab}^2 -r_{ab}^2)\Omega_{pq}(c,c)\label{test3}\end{align}
\end{lemma}
Eqs.~(\ref{test2}, \ref{test3}) can derived through the same analysis as one leading to Eq.~(\ref{test1}). Note that for each equation in Lemma \ref{LemmadiffsqLC}, the right side contains power covariance terms originating from the nodes in $D_a^{{\cal T}_k}$ alone. Thus, if the covariances of all descendants $c \neq a \in D_a^{{\cal T}_k}$ are known, Eqs.~(\ref{test1},\ref{test2},\ref{test3}) can be used to infer the three covariance quantities ($\Omega_p(a,a), \Omega_p(a,a),\Omega_{pq}(a,a)$) associated with node $a$. Furthermore, parameters ($r_{ab}, x_{ab}$) included in these equations pertain to the single operational line $(a,b)$. For the case where injection covariances $\Omega_p, \Omega_q$ are known from historical data, we can thus estimate the parameters of line $(a,b)$ as well as $\Omega_{pq}(a,a)$, the covariance between active and reactive injections at node $a$. We use these facts later in the text while designing our learning algorithms.

Next, we prove an important inequality involving the magnitude of $\mathbb{E}[(\varepsilon_a - \mu_{\varepsilon_a})-(\varepsilon_b-\mu_{\varepsilon_b})]^2$ on the grid nodes.

\begin{lemma} \label{Lemmacases}
For distinct nodes $a$, $b$ and $c$ that belong to the same tree ${\cal T}_k$, $\mathbb{E}[(\varepsilon_a - \mu_{\varepsilon_a})-(\varepsilon_b-\mu_{\varepsilon_b})]^2 < \mathbb{E}[(\varepsilon_a - \mu_{\varepsilon_a})-(\varepsilon_c-\mu_{\varepsilon_c})]^2$ holds for the following cases:
\begin{enumerate}
\item Node $a$ is a descendant of node $b$ and node $b$ is a descendant of node $c$ (see Fig.~\ref{fig:item1}),
\item Nodes $a$ and $c$ are descendants of node $b$ and the path from $a$ to $c$ passes through node $b$ (see Fig.~\ref{fig:item3}),
\end{enumerate}
\end{lemma}
\begin{figure}[!bt]
\centering
\subfigure[]{\includegraphics[width=0.20\textwidth]{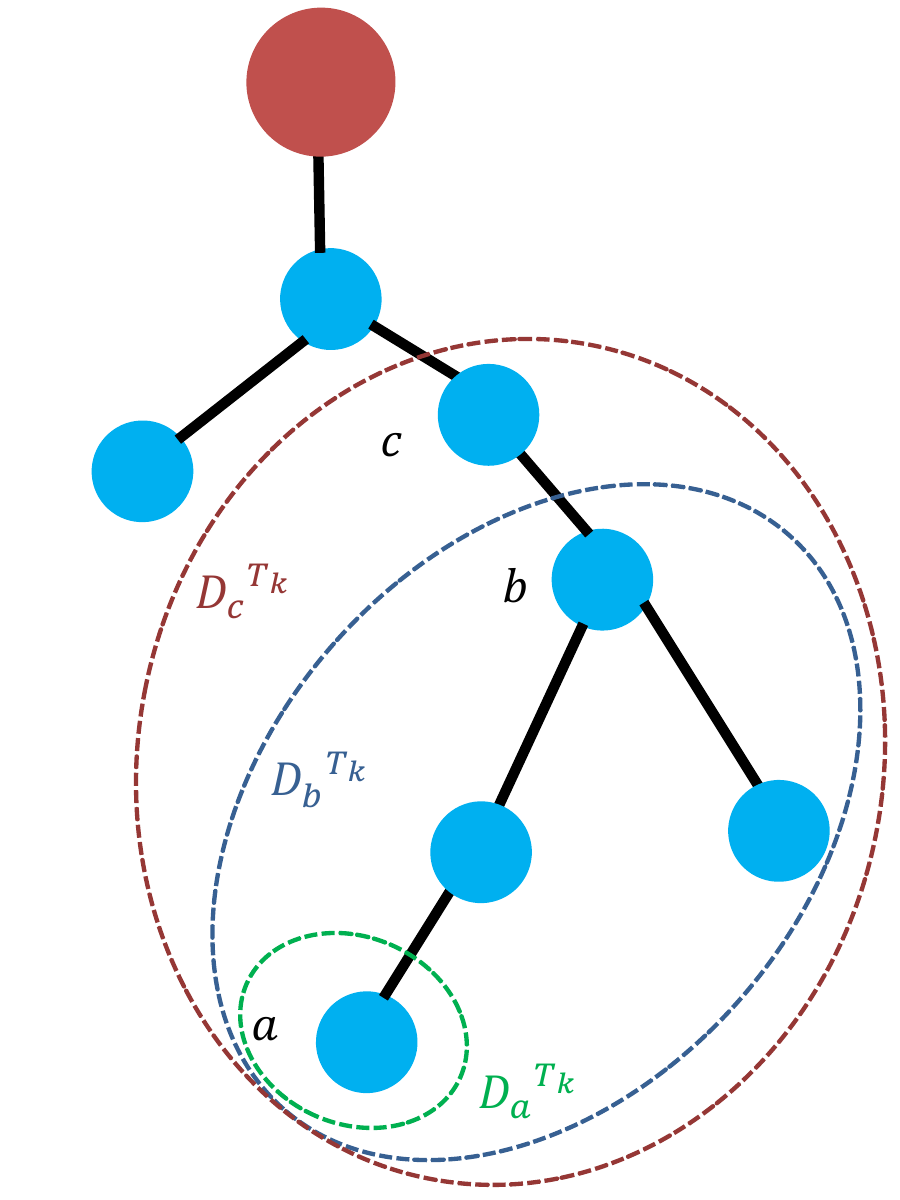}\label{fig:item1}}
\subfigure[]{\includegraphics[width=0.20\textwidth]{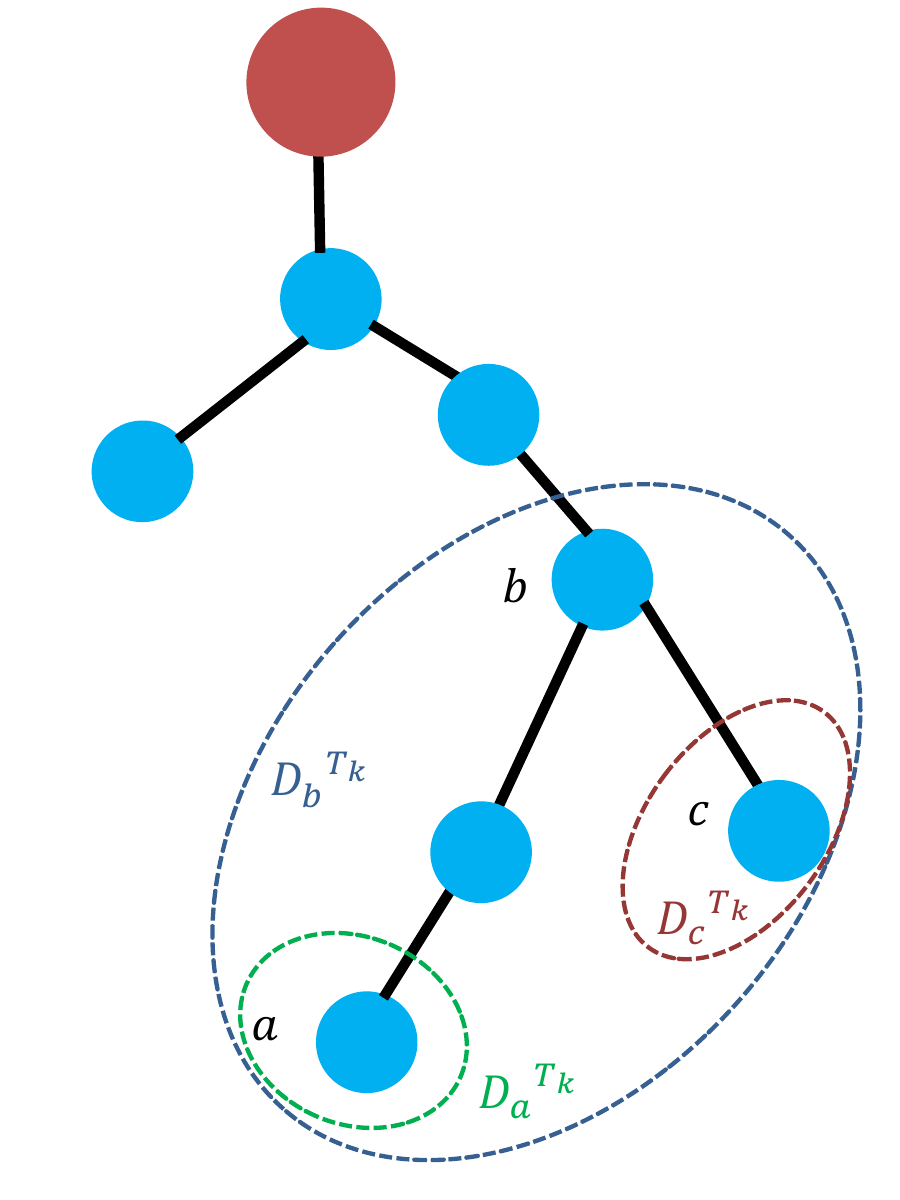}\label{fig:item3}}
\squeezeup
\caption{Schematic layout of a distribution grid tree ${\cal T}_k$. The sub-station node represented by large red node is the slack bus. $D_a^{{\cal T}_k}$ represents the set of nodes that are descendants of node $a$. (a) Node $a$ is a descendant of node $b$, while node $b$ is a descendant of node $c$. (b) Node $a$ and $c$ are descendants of node $b$ along disjoint sub-trees.
\label{fig:item}}
\end{figure}

\begin{proof}
Let us first prove the Lemma for the Case $1$. As shown in Fig.~\ref{fig:item1}, one observes
$D^{{\cal T}_k}_a \subseteq D^{{\cal T}_k}_b \subseteq D^{{\cal T}_k}_c$. Further, ${\cal E}_a^{{\cal T}_k}- {\cal E}_b^{{\cal T}_k} \subseteq {\cal E}_a^{{\cal T}_k}-{\cal E}_c^{{\cal T}_k}$, where ${\cal E}_a^{{\cal T}_k}$ represents edges traversed along the path leading from node $a$ to the root of ${\cal T}_k$. Consider a node $d$ in the tree ${\cal T}_k$. When $d \in D^{{\cal T}_k}_a$, one uses (\ref{Hrxinv}) to derive
\begin{align}
H^{-1}_{1/r}(a,d)-H^{-1}_{1/r}(b,d) &= \smashoperator[r]{\sum_{(ef) \in {\cal E}_a^{{\cal T}_k}-{\cal E}_b^{{\cal T}_k}}}r_{ef} < \smashoperator[r]{\sum_{(ef) \in {\cal E}_a^{{\cal T}_k}-{\cal E}_c^{{\cal T}_k}}}r_{ef}\nonumber\\
\Rightarrow ~H^{-1}_{1/r}(a,d)-H^{-1}_{1/r}(b,d) &< H^{-1}_{1/r}(a,d)-H^{-1}_{1/r}(c,d)\label{first}
\end{align}
Similarly, for node $d \in D^{{\cal T}_k}_b - D^{{\cal T}_k}_a$, one obtains
\begin{align}
&H^{-1}_{1/r}(a,d)-H^{-1}_{1/r}(b,d) = \smashoperator[lr]{\sum_{(ef) \in {\cal E}_a^{{\cal T}_k}\cap {\cal E}_{d}^{{\cal T}_k}-{\cal E}_b^{{\cal T}_k}}} r_{ef} ~~~~< \smashoperator[r]{\sum_{(ef) \in {\cal E}_a^{{\cal T}_k}\cap {\cal E}_{d}^{{\cal T}_k}-{\cal E}_c^{{\cal T}_k}} } r_{ef}\nonumber\\
\Rightarrow ~&H^{-1}_{1/r}(a,d)-H^{-1}_{1/r}(b,d) < H^{-1}_{1/r}(a,d)-H^{-1}_{1/r}(c,d)\label{second}
\end{align}
For $d \in D^{{\cal T}_k}_c - D^{{\cal T}_k}_b$, we arrives at
\begin{align}
&H^{-1}_{1/r}(a,d)-H^{-1}_{1/r}(b,d)= 0 < \smashoperator[r]{\sum_{(ef) \in {\cal E}_a^{{\cal T}_k}\cap {\cal E}_{d}^{{\cal T}_k}-{\cal E}_c^{{\cal T}_k}}} r_{ef}\nonumber\\
\Rightarrow~& H^{-1}_{1/r}(a,d)-H^{-1}_{1/r}(b,d) < H^{-1}_{1/r}(a,d)-H^{-1}_{1/r}(c,d)\label{third}
\end{align}
Next, using Eqs.~(\ref{first},\ref{second},\ref{third}), we arrive at
\begin{align}
&\forall d \in D^{{\cal T}_k}_c, H^{-1}_{1/r}(a,d)-H^{-1}_{1/r}(b,d) \leq H^{-1}_{1/r}(a,d)-H^{-1}_{1/r}(c,d)\label{combine}\\
&\forall d \not\in D^{{\cal T}_k}_c,H^{-1}_{1/r}(a,d)-H^{-1}_{1/r}(b,d),~ H^{-1}_{1/r}(a,d)-H^{-1}_{1/r}(c,d) = 0\label{combine1}
\end{align}
Similar inequalities hold for $H^{-1}_{1/x}$ as well. We can now apply Eqs.~(\ref{combine},\ref{combine1}) to Eq.~(\ref{usediff_1}) to prove $\mathbb{E}[(\varepsilon_a - \mu_{\varepsilon_a})-(\varepsilon_b-\mu_{\varepsilon_b})]^2 < \mathbb{E}[(\varepsilon_a- \mu_{\varepsilon_a})-(\varepsilon_c-\mu_{\varepsilon_c})^2]$ for Case $1$.

In the case $2$ (see Fig.~\ref{fig:item3}) nodes $a$ and $c$ are descendants of node $b$. Let $r_a$ be the penultimate (second to the last) node lying on the path from $a$ to $c$, and $r_c$ be the penultimate node on the path from $c$ to $b$. Here, $D^{{\cal T}_k}_{r_a}$ and $D^{{\cal T}_k}_{r_c}$ are disjoint subsets of $D^{{\cal T}_k}_b$. Then, for any $d_a \in D^{{\cal T}_k}_{r_a}$ and $d_c \in D^{{\cal T}_k}_{r_c}$, observe that ${\cal E}_{d_a}^{{\cal T}_k} \bigcap {\cal E}_{d_c}^{{\cal T}_k} = {\cal E}_b^{{\cal T}_k}$. This results in
\begin{align}
H^{-1}_{1/r}(b,d_a)-H^{-1}_{1/r}(a,d_a)&= H^{-1}_{1/r}(c,d_a)-H^{-1}_{1/r}(a,d_a) \label{first2}\\
H^{-1}_{1/r}(b,d_c)-H^{-1}_{1/r}(a,d_c)&=0 < H^{-1}_{1/r}(c,d_c)-H^{-1}_{1/r}(a,d_c)\label{second2}
\end{align}
Furthermore, for $d \not\in D^{{\cal T}_k}_{r_a} \bigcup D^{{\cal T}_k}_{r_c}$,
\begin{align}
H^{-1}_{1/r}(b,d)-H^{-1}_{1/r}(a,d)= 0 =H^{-1}_{1/r}(c,d)-H^{-1}_{1/r}(a,d) \label{third2}
\end{align}
Versions of Eqs.~(\ref{first2},\ref{second2},\ref{third2}) for $H^{-1}_{1/x}$ can be derived in a similar way. Using these results in Eq. \ref{usediff_1}, one arrives at $\mathbb{E}[(\varepsilon_a - \mu_{\varepsilon_a})-(\varepsilon_b-\mu_{\varepsilon_b})]^2 <\mathbb{E}[(\varepsilon_a- \mu_{\varepsilon_a})-(\varepsilon_c- \mu_{\varepsilon_c})]^2$ for Case $2$. This completes the proof.
\end{proof}

The following theorem follows directly from Lemma \ref{Lemmacases}.
\begin{theorem} \label{Theorem4}
For every node $a$ with set of descendants $D^{{\cal T}_k}_a$ and parent $b$, $b = \arg \min_{c \not\in D^{{\cal T}_k}_a} \mathbb{E}[(\varepsilon_a - \mu_{\varepsilon_a})-(\varepsilon_c-\mu_{\varepsilon_c})]^2.$
\end{theorem}
\begin{proof}
In the case $2$ of Lemma \ref{Lemmacases}, the optimal node for $\arg \min_{c \not\in D^{{\cal T}_k}_a} \mathbb{E}[(\varepsilon_a - \mu_{\varepsilon_a})-(\varepsilon_c-\mu_{\varepsilon_c})]^2$ exists on the path from node $a$ to the root. Considering case $1$, one finds that the optimal node on that path is node $a$'s parent $b$.
\end{proof}
Theorem \ref{Theorem4} implies that among all non-descendants of a node, the minimum expected squared centered difference of voltage magnitude deviations is achieved at its parent node. Indeed in the next Section, we utilize this result to identify a node's parent.

\section{Learning Grid Structure with Estimation of Load or Parameters}
\label{sec:algo1}
We first present our algorithm design for Tasks $1$ and $2$, structure learning coupled with estimation of nodal power injection statistics. Next, we look at solving for Tasks $1$ and $3$, structure learning coupled with estimation of line parameters.

\subsection{Learning Structure and Injection Statistics}
\label{subsec:task12}
The results of the previous Section (specifically, Theorem \ref{Theorem1_LC}, Lemma \ref{LemmadiffsqLC} and Theorem \ref{Theorem4}) provide the machinery for the algorithm design. Algorithm $1$ learns the radial operational structure (Task $1$) as well as estimates the mean $\mu_{p}$ and covariance $\Omega_{p}$ of the power injections at the load nodes (Task $2$). The observer here is assumed to be aware of the load nodes that are connected directly to the grid sub-stations. This is necessary as the assignment of substations, one per tree in forest $\cal F$ cannot be uniquely determined. This occurs due to the assumption of zero fluctuations of voltage magnitude and phase at substations which makes the relations involving voltage deviations in the previous section insensitive when the substation is the parent node. Resistance and reactance parameters of all lines (open and operational) are assumed known here.

\begin{algorithm}
\caption{Base Constrained Spanning Forest Learning with Estimation of Load Statistics:}
\textbf{Input:} $m$ phase angle and voltage deviation observations $\theta^j$ and $\varepsilon_i, 1\leq j \leq m$, all line resistances $r$ and line reactances $x$\\
\textbf{Output:} Covariance Matrices $\Omega_{p}$, $\Omega_q$ and $\Omega_{pq}$, mean vectors $\mu_{p}$ and $\mu_q$\\
\begin{algorithmic}[1]
\State Compute $\mu_{\theta_a} = \sum_{j = 1}^m\theta^j_a/m, \mu_{\varepsilon_a} = \sum_{j = 1}^m\varepsilon^j_a/m, \Omega_{\theta}(a,a) = \sum_{j = 1}^m\theta^j_a\theta^j_a/m-\mu^2_{\theta_a}$ and $\Omega_{\varepsilon}(a,a) = \sum_{j = 1}^m\varepsilon^j_a\varepsilon^j_a/m-\mu^2_{\varepsilon_a}$ for all nodes $a$.
\State Undiscovered Set $U \gets \{1,2,...,N\}$, Leaf Set $L \gets \phi$, Descendant Covariance vectors $D^p \gets \textbf{0}, D^q \gets \textbf{0}$, $D^{pq} \gets \textbf{0}$.
\While {($U \neq \phi)$}
\State $b^* \gets \max_{b \in U} \Omega_{\varepsilon}(b,b)$ \label{step3_1}
 \ForAll{$a \in L$}
\If {$b^* = \arg\min_{c \in U}\sum_{j = 1}^m[(\varepsilon^j_a - \mu_{\varepsilon_a})-(\varepsilon^j_c-\mu_{\varepsilon_c})]^2/m$} \label{step3_2}
\State Draw edge between nodes $a$ and $b^*$
\State Solve Eqs.~(\ref{test1},\ref{test2},\ref{test3}) to get $\Omega_{p}(a,a)$, $\Omega_{q}(a,a)$ and $\Omega_{pq}(a,a)$ \label{step3_3}
\State $D^p(b^*) \gets D^p(b^*) + \Omega_{p}(a,a) + D^p(a)$
\State $D^q(b^*) \gets D^q(b^*) + \Omega_{q}(a,a) + D^q(a)$
\State $D^{pq}(b^*) \gets D^{pq}(b^*) + \Omega_{pq}(a,a) + D^{pq}(a)$
\State $L \gets L - \{a\}$ \label{step3_4}
\EndIf
\EndFor
 \State $L \gets L \bigcup \{b^*\}$ \label{step3_5}
\EndWhile
\State Generate $H_{1/x}$ and $H_{1/r}$ from edges
\State Solve $\mu_{\theta} = H^{-1}_{1/x}\mu_p - H^{-1}_{1/r}\mu_q,~~ \mu_\varepsilon = H^{-1}_{1/r}\mu_p + H^{-1}_{1/x}\mu_q$ \label{step3_6}
\end{algorithmic}
\end{algorithm}

\textbf{Algorithm Overview:} In each iteration, the node $b^*$ with the highest variance in voltage deviation among the yet undiscovered node set $U$ is selected in Step \ref{step3_1}. Theorem \ref{Theorem1_LC} ensures that selecting nodes in the decreasing order of their variances leads to discovery of node $b^*$ only after all its descendants have been discovered previously. Set $L$ denotes the current set of leaves (previously discovered nodes with unknown parents). In Step \ref{step3_2}, the selected node $b^*$ is made the parent of a node in set $L$ if the condition in Theorem \ref{Theorem4} is satisfied. Here, each entry in the descendant covariance vectors $D^p, D^q$ and $D^{pq}$ contains the sum of load power covariances over all descendants of each node, other than the node itself. The values of covariance matrices of power injections for $b^*$ are inferred in Step \ref{step3_3} using Lemma \ref{LemmadiffsqLC}. Steps \ref{step3_4} and \ref{step3_5} are used to update the current set of leaves $L$ for use in the next iteration. Finally, in Step \ref{step3_6}, the mean of nodal power is computed using the measurement matrix $H$ constructed from the grid structure. Note that instead of learning the covariances in $\Omega_P$ sequentially through Step \ref{step3_3}, one can use the generated measurement matrix $H$ directly to learn them together at the end. \\

\textbf{Algorithm Complexity:} Computing empirical covariance matrix of voltage deviation is considered to be a part of pre-processing and is thus ignored in the complexity estimation. One makes $N$ iterations to select all the non-substation nodes. Within each iteration, an edge selection (Step \ref{step3_2}) calls for a check with each node in $L$. Thus, the worst-case complexity for learning the structure is $O(N^2)$. Computing the means and the covariances is of complexity $O(N^2)$ through matrix multiplication.

Observe that learning the forest structure in Algorithm $1$ relies on voltage magnitude deviation measurements alone, and in fact does not require knowledge of line parameters in the grid. Phase measurements and values of line resistance and reactance are needed only to estimate the means and covariances of power injections.

\subsection{Learning Structure and Line Parameters}
The first goal of the observer here is the same as in Section \ref{subsec:task12} -  to learn the operational grid structure. However, we consider a modified scenario where the covariances for active and reactive nodal injections ($\Omega_p$ and $\Omega_q$) are already known from historical data and thus do not need to be estimated. Instead, the observer here aims at estimating the impedance parameters ($r_{ab}$ and $x_{ab}$) for each operational line $(a,b)$ within the grid. Consider Eqs.~(\ref{test1},\ref{test2},\ref{test3}). If matrix $\Omega_{pq}$ is also known, the observer can easily solve these linear equations with $r_{ab}$, $x_{ab}$ and $(r_{ab}x_{ab})$ as the three unknowns to estimate the impedance for each operational edge. However, $\Omega_{pq}$ may be harder to obtain in reality as its computation requires time-synchronized historical samples of active and reactive injections. If information on $\Omega_{pq}$ is unavailable, variables $r_{ab}, x_{ab}$ and $\Omega_{pq}(a,a)$ form three nonlinear Eqs.~(\ref{test1},\ref{test2},\ref{test3}) for each edge $(a,b)$. Note that Algorithm $1$ infers the radial grid structure iteratively upward from the descendant nodes to the parents. Therefore, we also infer line parameters ($r_{ab}, x_{ab}$) and $\Omega_{pq}(a,a)$ by solving Eqs~(\ref{test1},\ref{test2},\ref{test3}) in each iteration for the newly discovered edge $(a,b)$ between node $a$ and its parent $b$ in tree ${\cal T}_k$. Let $A$, $B$ and $C$ denote the expressions on the left side of Eqs.~(\ref{test1},\ref{test2},\ref{test3}) respectively. From  Eqs.~(\ref{test1},\ref{test2}), we derive, $r_{ab}^2 +x_{ab}^2 = \left(A + B\right)/\left(\sum_{c \in D^{{\cal T}_k}_a}\Omega_p(c,c) + \Omega_q(c,c)\right)$. We can now eliminate terms involving $x_{ab}$ and $\Omega_{pq}$ to get Eq.~(\ref{quadratic}) which is a quadratic expression in $r_{ab}^2$. We use it to infer $r_{ab}$ and $x_{ab}$. To infer $\Omega_{pq}(a,a)$, we use values of $\Omega_{pq}(c,c)$ for descendants $c (\neq a)$ of node $a$ that are determined in previous iterations.
\begin{table*}[bt]
\begin{align}
r_{ab}^4\left((A -B)^2 +4C^2\right) + \frac{(A +B)^2(A\smashoperator[lr]{\sum_{c \in D^{{\cal T}_k}_a}}\Omega_p(c,c)-B\smashoperator[lr]{\sum_{c \in D^{{\cal T}_k}_a}}\Omega_q(c,c))^2}{(\smashoperator[lr]{\sum_{c \in D^{{\cal T}_k}_a}}\Omega_p(c,c)+\Omega_q(c,c))^4}= 2r^2\frac{(A\smashoperator[lr]{\sum_{c \in D^{{\cal T}_k}_a}}\Omega_p(c,c)-B\smashoperator[lr]{\sum_{c \in D^{{\cal T}_k}_a}}\Omega_q(c,c))(A^2-B^2) + 2C^2(A+B)(\smashoperator[lr]{\sum_{c \in D^{{\cal T}_k}_a}}\Omega_p(c,c)+\Omega_q(c,c))}{(\smashoperator[lr]{\sum_{c \in D^{{\cal T}_k}_a}}\Omega_p(c,c)+\Omega_q(c,c))^2} \label{quadratic}
\end{align}
\end{table*}

Every step in this algorithm, except modified Step \ref{step3_3}, corresponds to respective step in Algorithm $1$. The Step \ref{step3_3} is modified such that Eqs.~(\ref{quadratic}), followed from (\ref{test1}), are used to derive the line parameters and $\Omega_{pq}$. As this algorithm formulation and analysis follows Algorithm $1$, we omit it for brevity. In the next Section, we discuss a critical extension of the structure learning problem (Task $1$) to the case where the available nodal data is incomplete due to some missing entries.

\section{Learning Base-Constrained Spanning Forest with Missing Data}
\label{sec:missing}
The structure learning problem discussed in the preceding Section (Task (1)) requires the observer to have voltage magnitude data for all nodes within the distribution grid. However, this may not be true in practice. In fact, loss of communication and/or synchronization troubles with meters over short periods of time, along with meter breakdowns over longer time-scales, can result in missing data over a set ${\cal M}$ of missing nodes in the system. We assume here that the ``missing" nodes are positioned within the grid not fully arbitrarily, but they satisfy the following property.

\textbf{Assumption $2$:} Missing nodes in set ${\cal M}$ are separated by greater than two hops in the distribution grid forest and they are not immediate children (not first descendants) of the sub-station nodes.
\begin{figure}[!bt]
\centering
\subfigure[]{\includegraphics[width=0.20\textwidth]{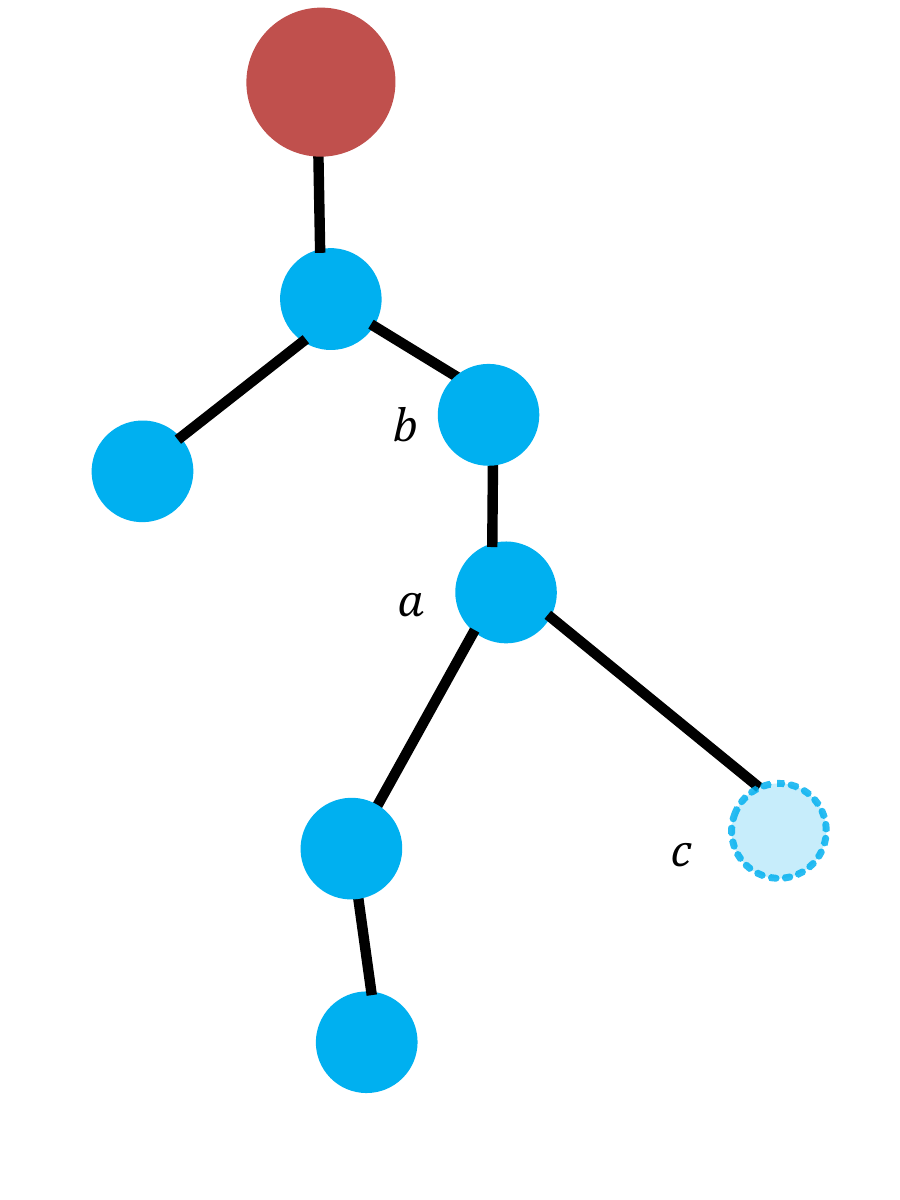}\label{fig:missing1}}
\subfigure[]{\includegraphics[width=0.20\textwidth]{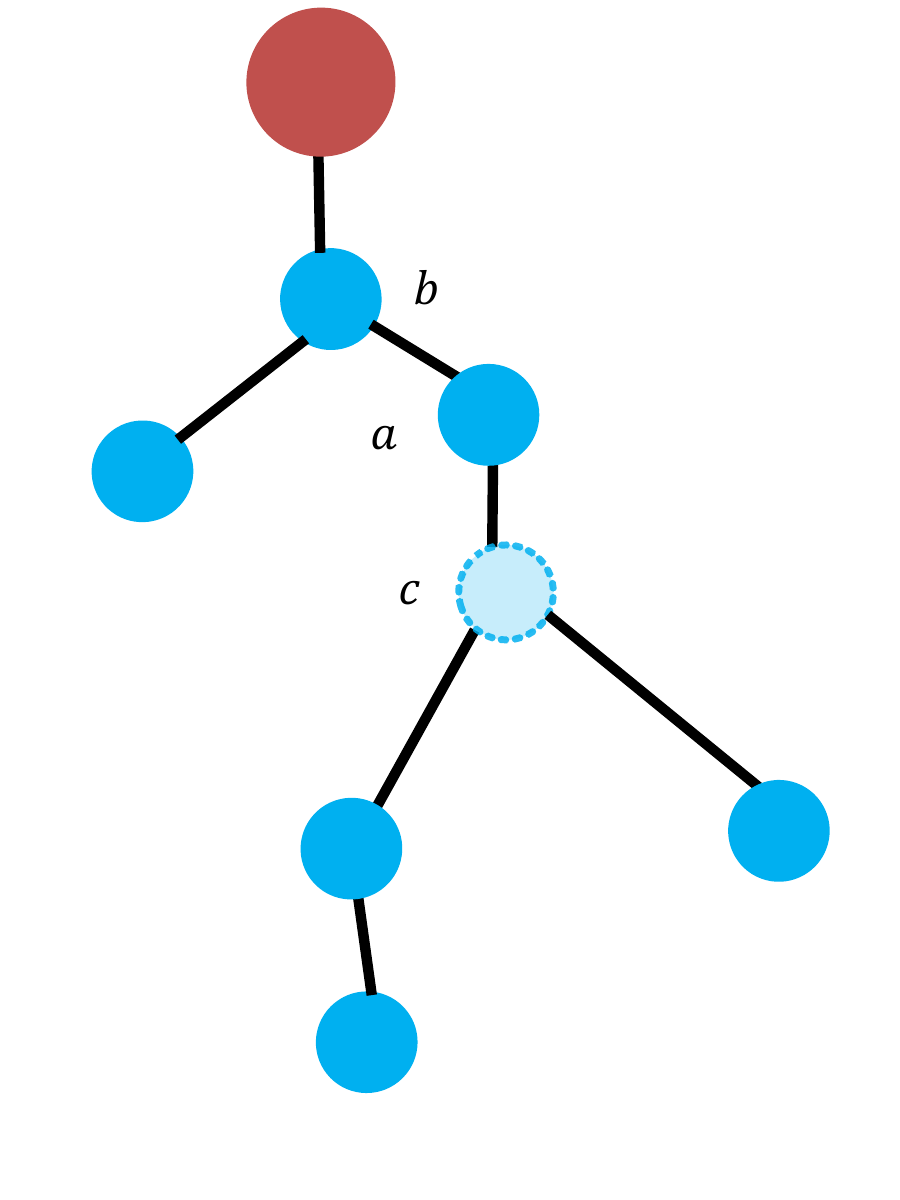}\label{fig:missing2}}
\squeezeup
\caption{Schematic layout of a distribution grid tree ${\cal T}_k$ with missing node $c$. The sub-station node, shown as the large red node, is the slack bus. (a) Missing node $c$ is a leaf with parent $a$. (b) Missing node $c$ is an intermediate node with parent node $a$ and grandparent node $b$.
\label{fig:missing}}
\end{figure}
This assumption implies that there exists no observed node which is connected to more than one missing node. Note that a missing node can exist in either of the two possible configurations -  a leaf or an intermediary position - as illustrated in Fig.~\ref{fig:missing}. Assumption $2$ guarantees that in either case, both the parent and grandparent (parent of the parent) nodes of the missing node are observed. Additionally, unlike structure learning in Task (1), in this section we assume that information, e.g. estimated or originating from historical measurements, on the actual values of $\Omega_p, ~\Omega_q$ and $\Omega_{pq}$ covariance matrices  and impedances of all lines is available. We now construct Algorithm $2$ to learn the operational grid structure in the presence of a missing set ${\cal M}$ with nodes whose voltage magnitude deviations are unknown.
\begin{algorithm*}[t]
\caption{Base Constrained Spanning Forest Learning with Missing Data}
\textbf{Input:} True $\Omega_p$ and $\Omega_q$, $m$ voltage deviation observations $\varepsilon_i, 1\leq j \leq m$ for nodes in set ${\cal M}$, all line resistances $r$ and line reactances $x$, Missing nodes Set: ${\cal M}$\\
\begin{algorithmic}[1]
\State Compute $\mu_{\varepsilon_a} = \sum_{j = 1}^m\varepsilon^j_a/m$, and $\Omega_{\varepsilon}(a,a) = \sum_{j = 1}^m\varepsilon^j_a\varepsilon^j_a/m-\mu^2_{\varepsilon_a}$ for all observed nodes $a$.
\State Undiscovered Set $U \gets \{1,2,...,N+K\}$, Leaf Set $L \gets \phi$, Unconnected Descendant Sets $D_a \gets \phi \forall$ nodes $a$, Child Active and reactive Covariance vectors $D^p \gets \textbf{0}, ~D^q \gets \textbf{0}$
\While {($U \neq \phi)$}
 \State $b^* \gets \max_{b \in U} \Omega_{\varepsilon}(b,b)$ \label{step6_1}
   \ForAll{$a \in L$}
       \If {$b^* = \arg\min_{c \in U}\sum_{j = 1}^m[(\varepsilon^j_a - \mu_{\varepsilon_a})-(\varepsilon^j_c-\mu_{\varepsilon_c})]^2/m$} \label{step6_2}
     \If {$D_a =\phi$}
            \If {$\smashoperator[lr]{\sum_{j=1}^m}\frac{[(\varepsilon^j_a-\mu_{\varepsilon_a})-(\varepsilon^j_{b^*}-\mu_{\varepsilon_{b^*}})]^2}{m} = r_{ab}^2(\Omega_p(a,a)+D^p(a)) + x_{ab}^2(\Omega_q(a,a) +D^q(a))+ 2r_{ab}x_{ab}(\Omega_{pq}(a,a)$\\
            \hfill $+D^{pq}(a))$}  \label{step6_3}
                \State Draw edge between nodes $a$ and $b^*$
                \State $D^p(b^*) \gets D^p(b^*) + \Omega_{p}(a,a) + D^p(a),~D^q(b^*) \gets D^q(b^*) + \Omega_{q}(a,a) + D^q(a)$
                \State $D^{pq}(b^*) \gets D^{pq}(b^*) + \Omega_{pq}(a,a) + D^{pq}(a)$
                \State $L \gets L - \{a\}$
            \Else
                \If {$\exists d \in {\cal M}$ such that $\sum_{j=1}^m[(\varepsilon^j_a-\mu_{\varepsilon_a})-(\varepsilon^j_{b^*}-\mu_{\varepsilon_{b^*}})]^2/m = x_{ab}^2(\Omega_p(a,a)+D^p(a)+ \Omega_p(d,d))$ \\
               \hfill $+ r_{ab}^2(\Omega_q(a,a) +D^q(a)+ \Omega_q(d,d)) + 2r_{ab}x_{ab}(\Omega_{pq}(a,a) +D^{pq}(a)+\Omega_{pq}(d,d))$}
                    \State Draw edges between nodes $a$ and $b^*$, and $a$ and $d$
                    \State $D^p(b^*) \gets D^p(b^*) + \Omega_{p}(a,a) + D^p(a)+ \Omega_{p}(d,d),~D^q(b^*) \gets D^q(b^*) + \Omega_{q}(a,a) + D^q(a)+ \Omega_{q}(d,d)$
                    \State $D^{pq}(b^*) \gets D^{pq}(b^*) + \Omega_{pq}(a,a) + D^{pq}(a)+ \Omega_{pq}(d,d)$
                    \State $L \gets L - \{a\}$, ${\cal M} \gets {\cal M} - \{d\}$  \label{step6_4}
                 \Else
                     \State $D_{b^*} \gets D_{b^*}  \cup \{a\}, ~D^p(b^*) \gets D^p(b^*) + \Omega_{p}(a,a) + D^p(a),~D^q(b^*) \gets D^q(b^*) + \Omega_{q}(a,a) + D^q(a)$
                     \State $D^{pq}(b^*) \gets D^{pq}(b^*) + \Omega_{pq}(a,a) + D^{pq}(a)$
                     \State $L \gets L - \{a\}$
                 \EndIf
             \EndIf
        \Else
             \State Find $d \in {\cal M}$ such that $\sum_{j=1}^m[(\varepsilon^j_a-\mu_{\varepsilon_a})-(\varepsilon^j_{b^*}-\mu_{\varepsilon_{b^*}})]^2/m = x_{ab}^2(\Omega_p(a,a)+D^p(a)+ \Omega_p(d,d))$ \label{step6_5}
             \State \hfill $+ r_{ab}^2(\Omega_q(a,a) +D^q(a)+ \Omega_q(d,d))+ 2r_{ab}x_{ab}(\Omega_{pq}(a,a) +D^{pq}(a)+\Omega_{pq}(d,d))$
                    \State Draw edges between nodes $a$ and $b^*$, and nodes in $D_a$ and $d$
                    \State $D^p(b^*) \gets D^p(b^*) + \Omega_{p}(a,a) + D^p(a)+ \Omega_{p}(d,d),~D^q(b^*) \gets D^q(b^*) + \Omega_{q}(a,a) + D^q(a)+ \Omega_{q}(d,d)$
                    \State $D^{pq}(b^*) \gets D^{pq}(b^*) + \Omega_{pq}(a,a) + D^{pq}(a)+ \Omega_{pq}(d,d)$
                    \State $L \gets L - \{a\}$, ${\cal M} \gets {\cal M} - \{d\}$
         \EndIf
         \EndIf
    \EndFor
    \State $L \gets L \bigcup \{b^*\}$ \label{step6_6}
\EndWhile
\end{algorithmic}
\end{algorithm*}

\textbf{Algorithm Overview:} The construction of each operational tree begins by picking node $b^*$ with the largest value of covariance in the voltage deviation (Step \ref{step6_1}) and then advancing along the Algorithm sequentially. Here, the current leaf set $L$ denotes the set of discovered nodes with yet unknown parents. For every node $a$ in $L$, we observe its set of unconnected descendants $D_a$. Here $D_a$ is empty if all of $a$'s non-leaf children (immediate descendants) are known and have been linked to it. Note that $a$ may be a parent to a missing leaf node despite $D_a$ being empty. Thus, if $D_a$ is empty, first Step \ref{step6_3} checks if the selected node $b^*$ is the parent to node $a$ with all children discovered by using Eq.~\ref{test1}. If no link is found, then Step \ref{step6_4} checks if node $b^*$ and node $a$ are connected with a missing leaf node $c$ linked to $a$ in the configuration shown in Fig.~\ref{fig:missing1}. If still no link is found, the Algorithm stores $a$ as an unconnected descendant of $b^*$ in $D_{b^*}$. On the other hand, if $D_a$ is non-empty, the algorithm confirms, in Step \ref{step6_5}, existence of a missing intermediate node $c$ with parent node $a$ and grandparent node $b^*$ in the configuration shown in Fig.~\ref{fig:missing2}. One of these three checks is guaranteed to find an edge due to Assumption $3$. Following this, a new node is selected in the next iteration. The Algorithm completes when the set $U$ becomes empty. Since no child (immediate descendant) of substation nodes are missing (Assumption $3$), $U = \phi$ implies inclusion of all the missing nodes into the grid structure (${\cal M} = \phi$).

\textbf{Algorithm Complexity:} Following the complexity analysis of Algorithm $1$, we estimate,  counting number of possible comparisons, that the worst case complexity of the Algorithm $2$ is $O((N-|\cal M|)^2|\cal M|)$.

\section{Experiments}
\label{sec:experiments}
We test the performance of Algorithms $1$ and $2$ on three distribution grid test cases \cite{radialsource} listed in Table \ref{table_testcases} and described in detail in Part I \cite{distgridpart1}.
\begin{table}[ht]
\caption{Summary of the tested distribution grids}
\begin{center}
\begin{tabular}{|p{1cm}|p{3cm}|p{2cm}|p{1cm}|}
\hline
Test Case & Number of buses / substations / tie-switches & Additional Non-operational lines & Source\\
\hline
$bus\_13\_3$ & $13/3/3$ & $10$ & \cite{testcase1} \\\hline
$bus\_29\_1$ & $29/1/1$ & $20$ & \cite{testcase2} \\\hline
$bus\_83\_11$ & $83/11/13$ & $30$ & \cite{testcase3} \\\hline
\end{tabular}
\end{center}
\label{table_testcases}
\end{table}

For each experiment here, we pick an operational spanning forest layout $\cal F$ from the loopy grid graph $\cal G$ of a test system by opening the additional lines as well as the tie-switches. For this configuration, we choose statistics of consumption at each load node (we consider Gaussian for all experiments) and use it to generate multiple samples of nodal power injection. For each vector-valued sample, we fix voltages at the substations and run power flows to derive voltage magnitudes and phases at every node. Then, we compute empirical correlation functions of phases and voltages, averaging over all the generated samples. Finally, a valid observation set is created by hiding all the operational information other than what is required as input. Then, we run our algorithms and compare the resulting reconstruction with the actual operational case.

We start by simulating Algorithm $1$. For brevity, we present results on learning the grid structure with inference of load statistics (and not inference of line impedance parameters). Here, the observer has access to phases and voltage magnitudes at all the nodes as input. As noted in Table \ref{table1}, voltage magnitudes are sufficient for reconstructing the grid structure, but inference of load statistics require both voltage magnitude and phase measurements. Fig.~\ref{fig:algo4mean} and Fig.~\ref{fig:algo4covar} show the change in the average fractional errors for estimating means and covariance of the nodal injections with increasing number of samples for the three test systems considered. For both estimated quantities, the fractional errors are stated in terms of the difference between the true and estimated values relative to the true values. It is clear from the Figures that the average fractional error decays exponentially with the number of samples. Comparing Fig.~\ref{fig:algo4adj} with Figs.~\ref{fig:algo4mean} and \ref{fig:algo4covar}, we see that the number of samples required to accurately reconstruct the topology is much less than for reconstructing the nodal power distributions. Only when the number of samples is less than $100$ does the reconstruction of the topology begin to suffer.

\begin{figure}[!bt]
\centering
\subfigure[]{\includegraphics[width=0.42\textwidth,height = .35\textwidth]{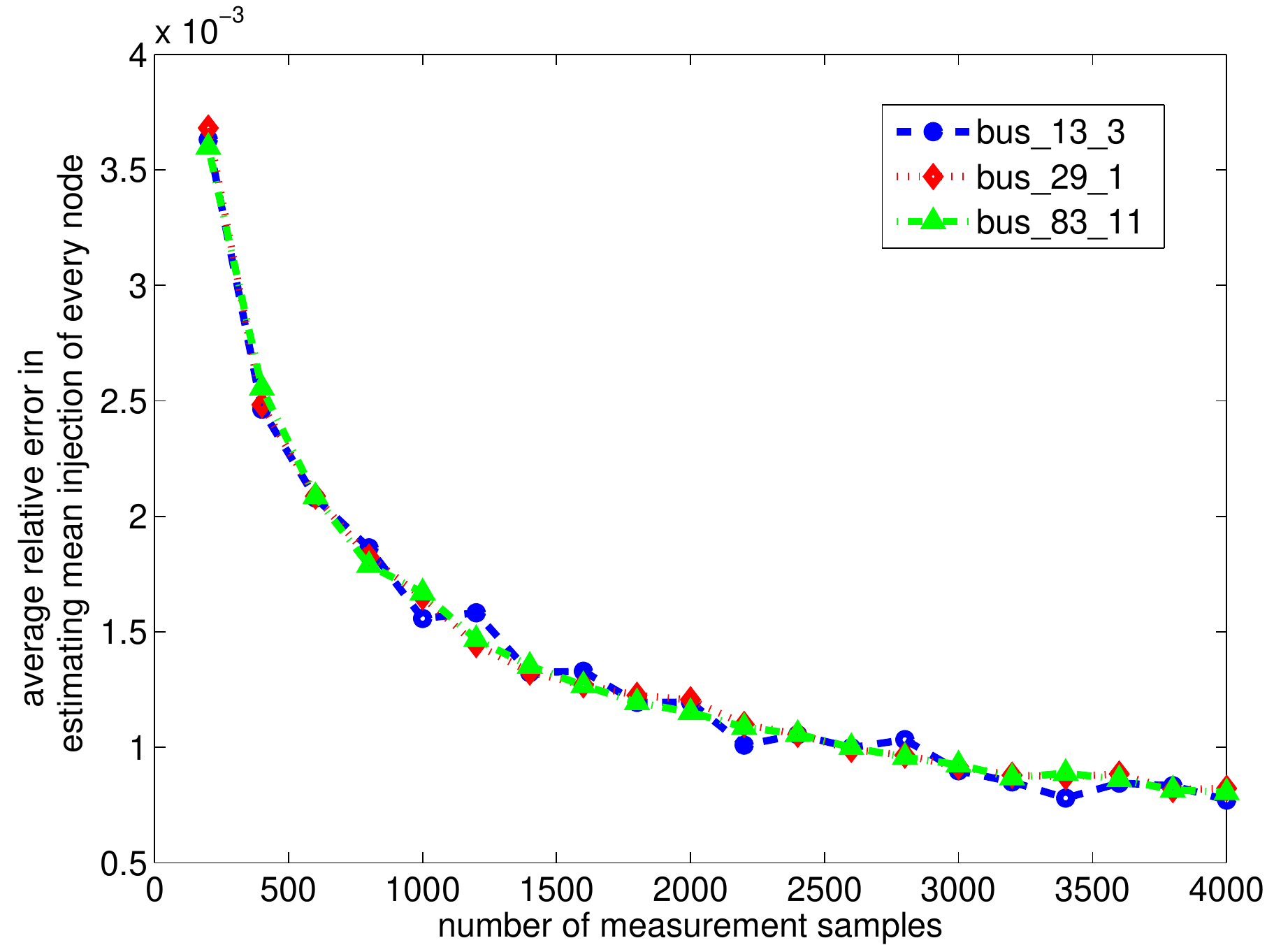}\label{fig:algo4mean}}
\subfigure[]{\includegraphics[width=0.42\textwidth,height = .35\textwidth]{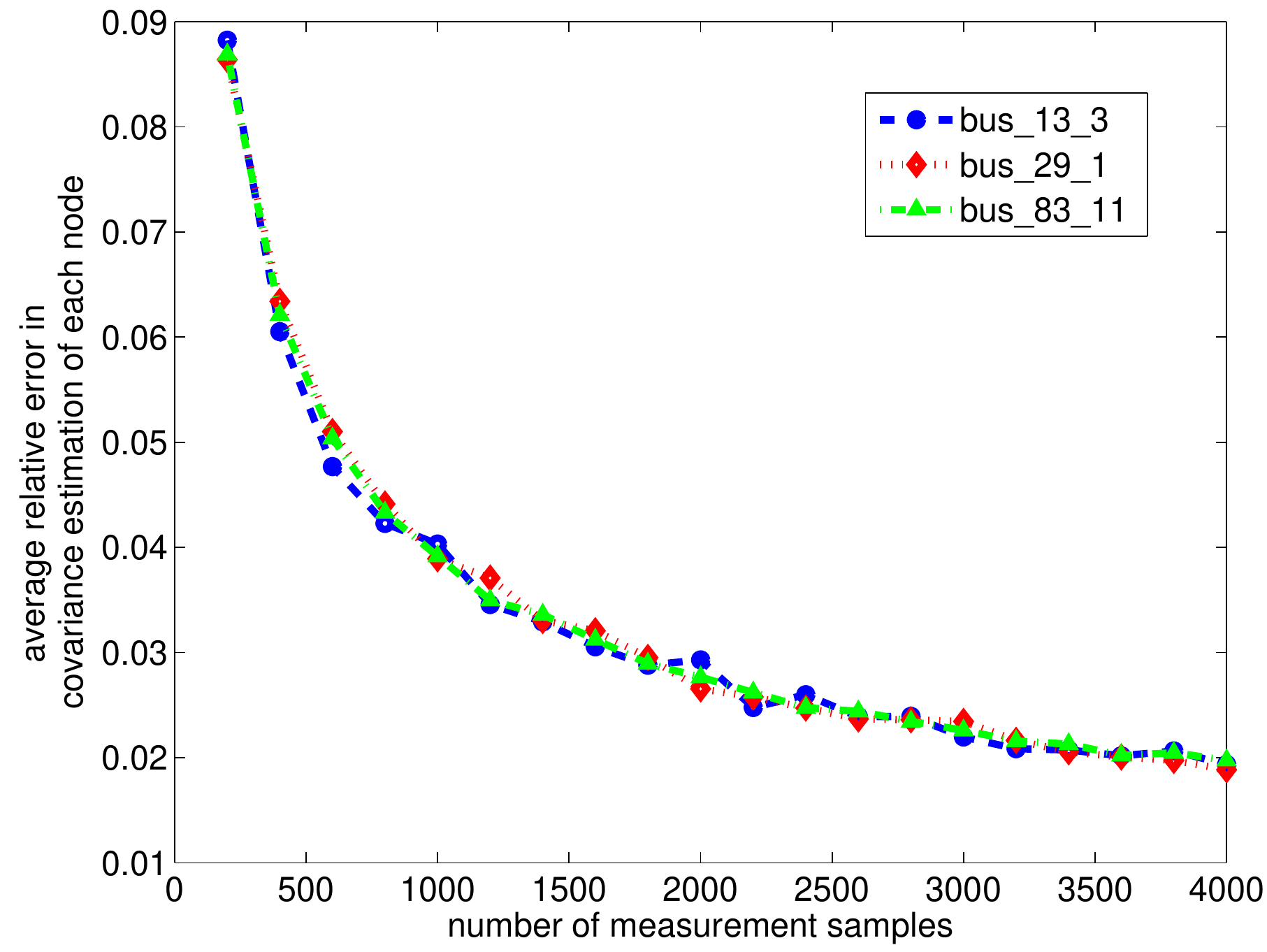}\label{fig:algo4covar}}
\subfigure[]{\includegraphics[width=0.42\textwidth,height = .35\textwidth]{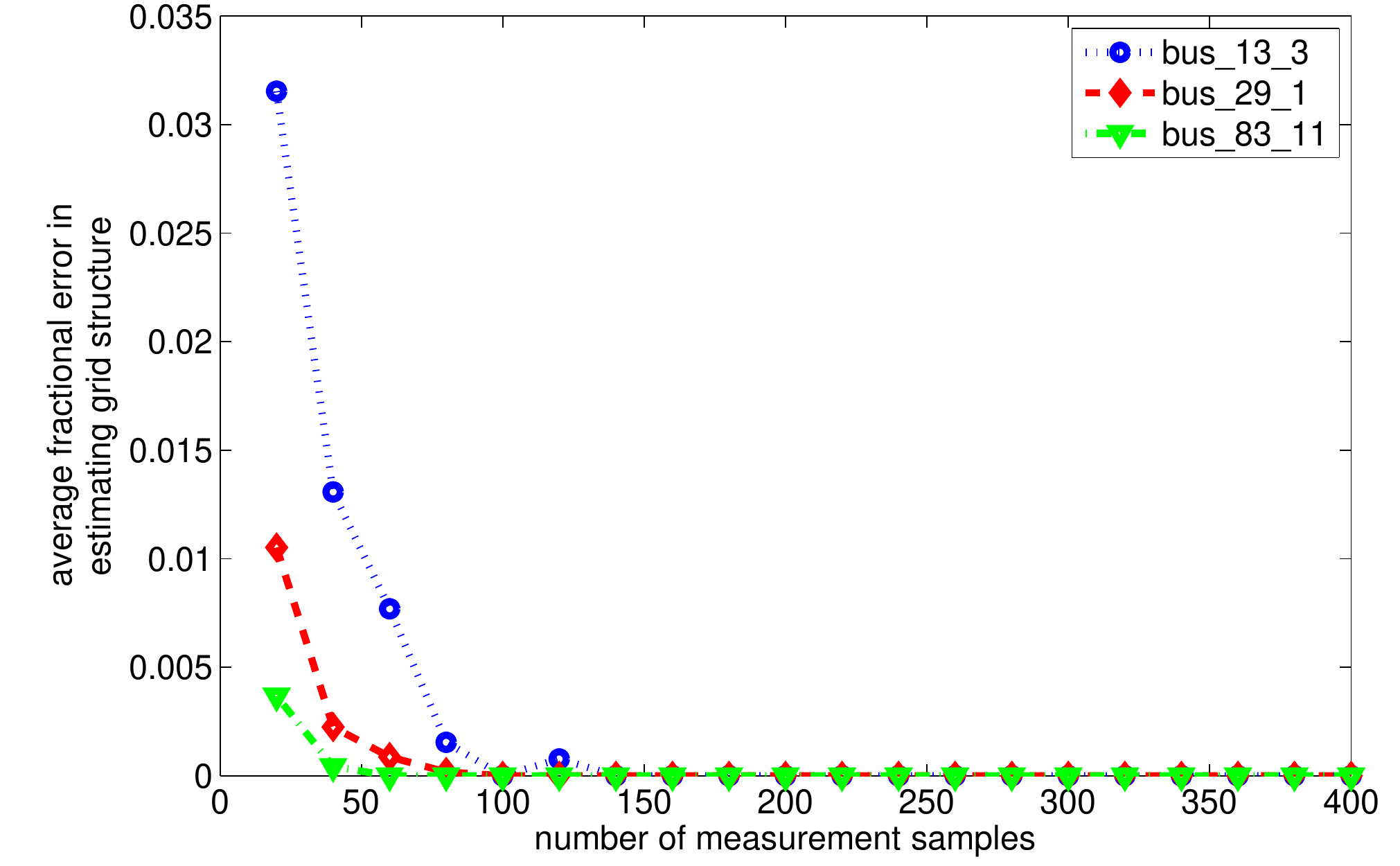}\label{fig:algo4adj}}
\squeezeup
\caption{Average fractional errors vs number of samples used in Algorithm $1$ for learning statistics of nodal injections and grid structure using Algorithm $1$. (a) Means of nodal power injection. (b) Covariances of nodal power injection. (c) Grid (forest/tree) reconstruction. The number of samples used for graph (forest/tree) reconstruction is moderate in comparison to the numbers used to estimate statistics.
\label{fig:algo4}}
\end{figure}

Next, we turn to discussing Algorithm $2$ that learns the grid structure from voltage magnitude measurements at a subset of the grid nodes. The actual covariance of the nodal injections of active and reactive powers is assumed known to the observer in this case. As described previously, we generate samples of the active and reactive injections, run power flows to generate samples of voltage magnitudes, but then erase samples before passing them to the observer. We study average fractional errors in learning the grid structure as a function of number of measurement samples. Note that averaging here is over both selection of the missing nodes and statistics of nodal injections. Fig.~\ref{fig:algo6case1}, Fig.~ \ref{fig:algo6case2} and Fig.~\ref{fig:algo6case3} show results for $bus\_13\_3$, $bus\_29\_1$ and $bus\_83\_11$ models respectively. Different curves within each Figure are generated using different number of missing nodes. As expected, the number of errors increases with increase in the number of missing nodes. The decay in the average fractional errors is exponential with increase in the number of samples.
\begin{figure}[!bt]
\centering
\subfigure[]{\includegraphics[width=0.42\textwidth,height = .36\textwidth]{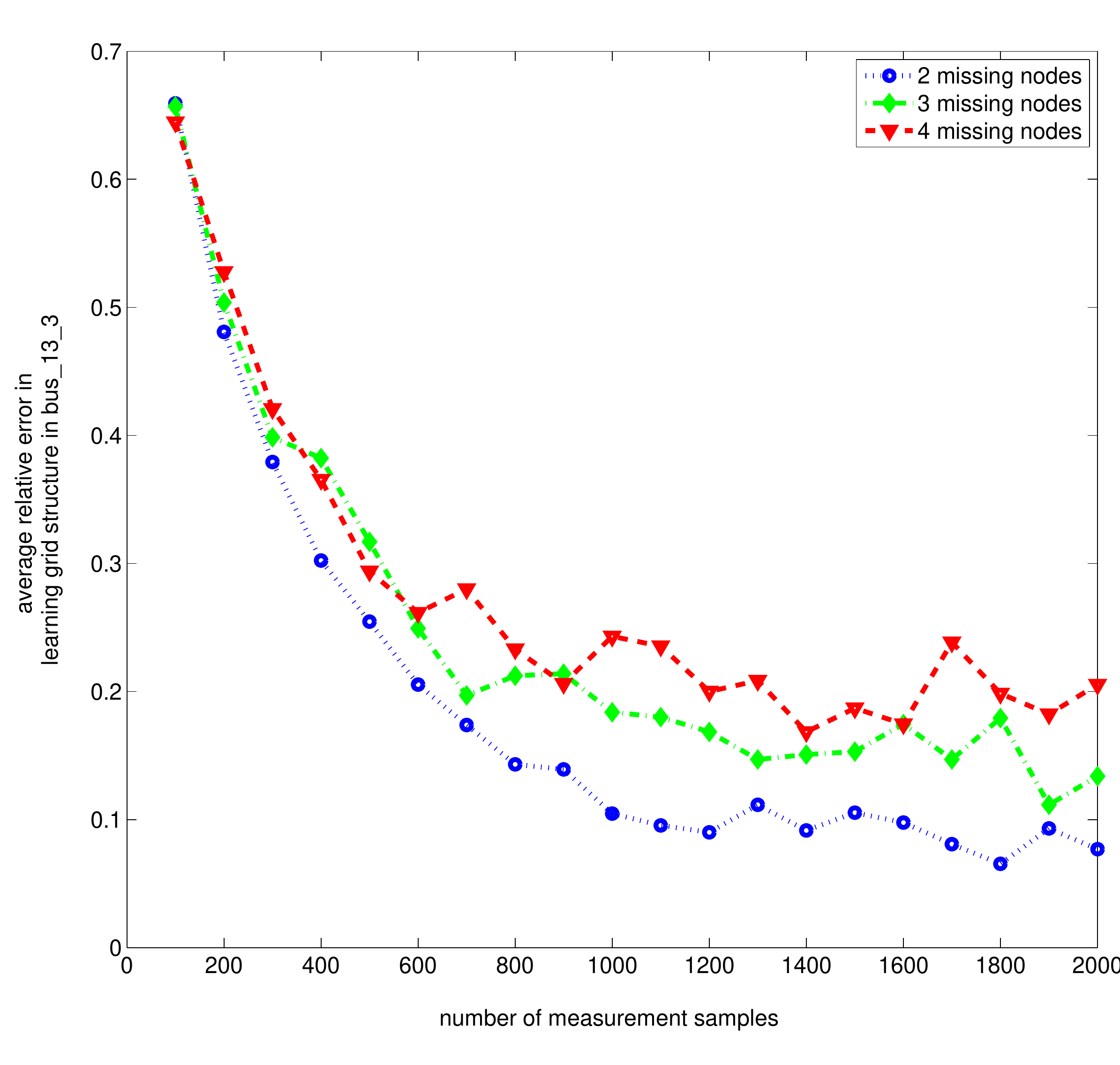}\label{fig:algo6case1}}
\subfigure[]{\includegraphics[width=0.42\textwidth,height = .37\textwidth]{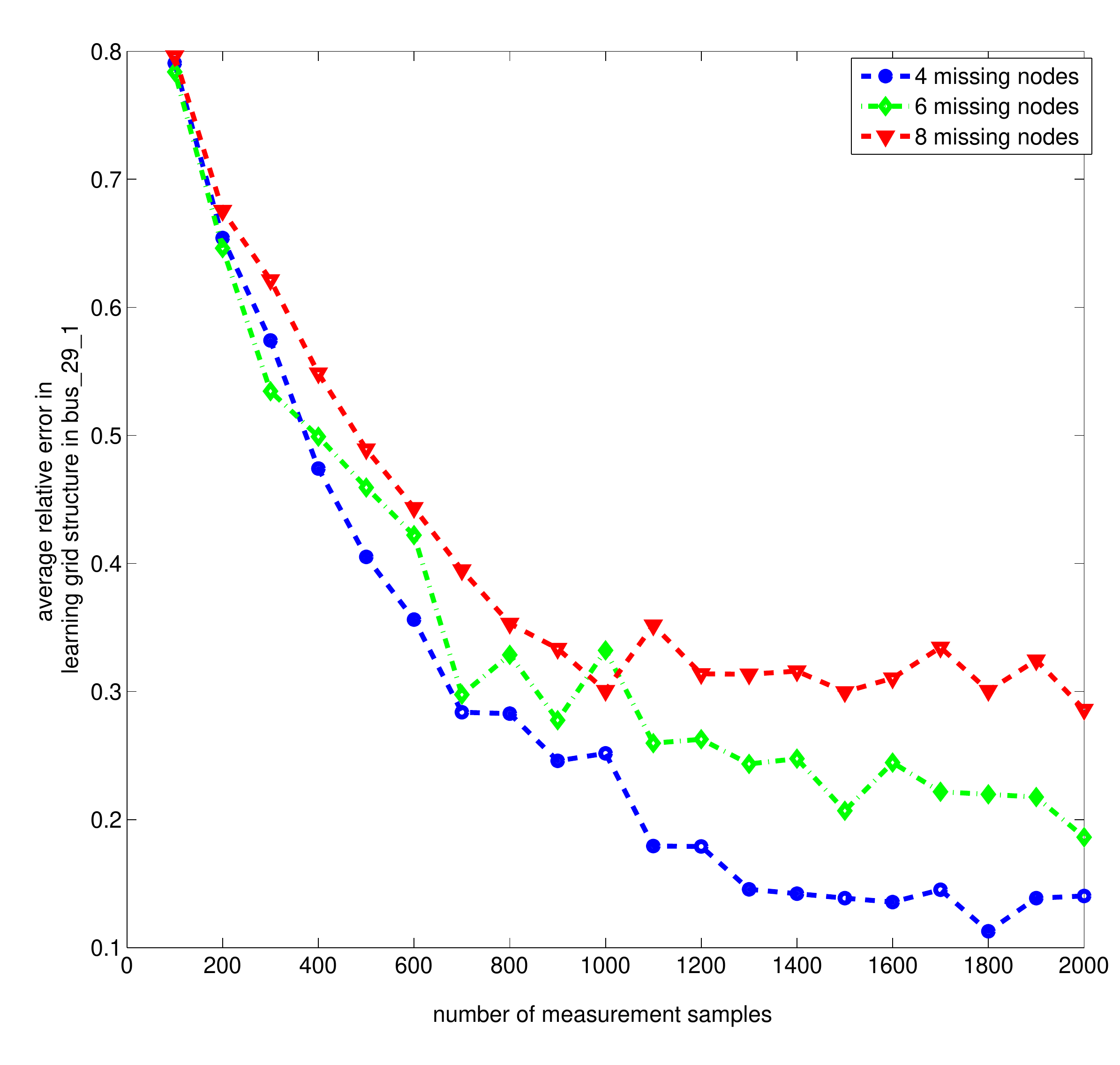}\label{fig:algo6case2}}
\subfigure[]{\includegraphics[width=0.42\textwidth,height = .37\textwidth]{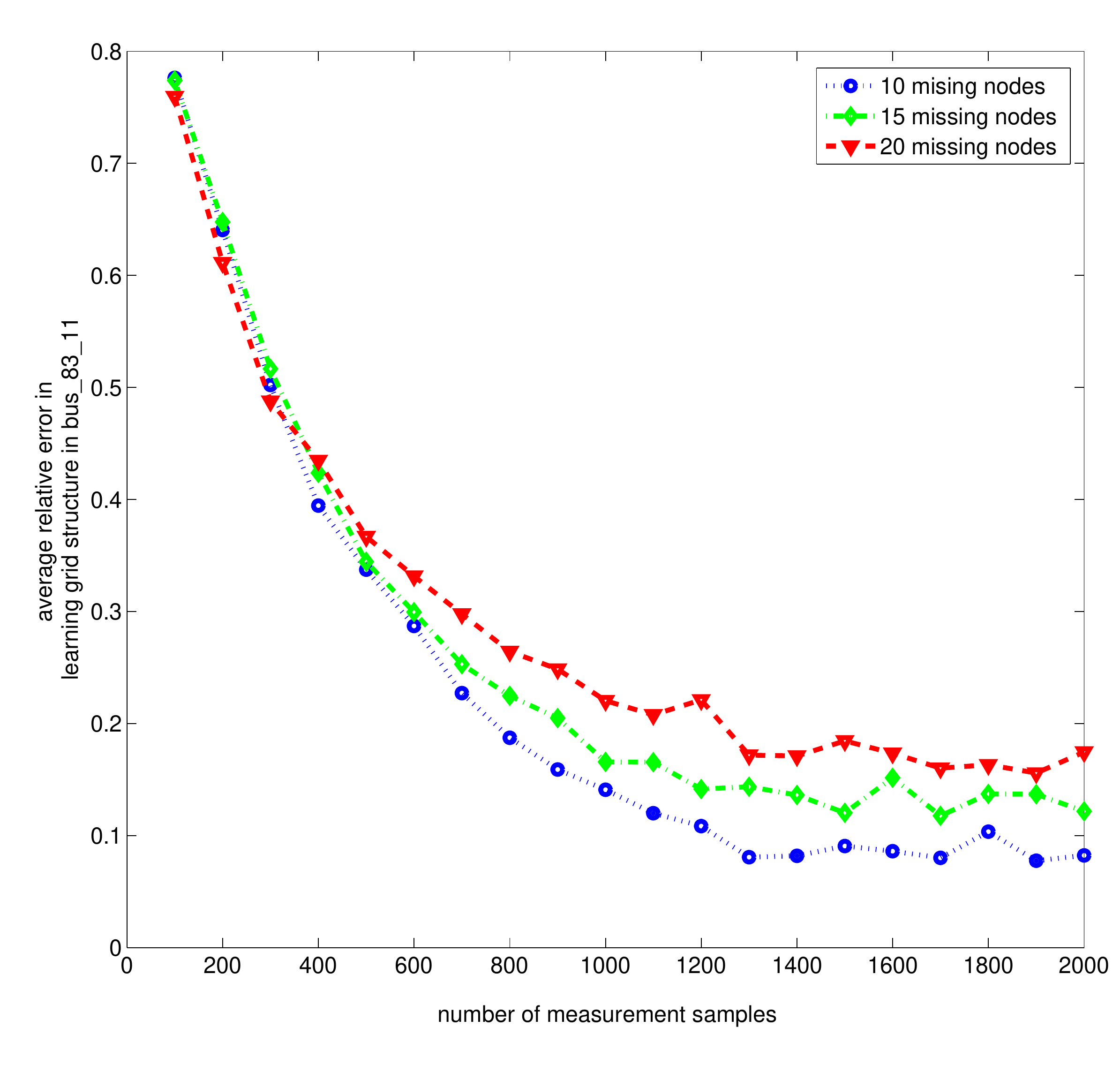}\label{fig:algo6case3}}
\squeezeup
\caption{Accuracy of Algorithm $2$ in learning the distribution grid structure vs number of samples in the presence of missing data for the test cases of (a) model $bus\_13\_3$, (b) model $bus\_29\_1$, and (c) model $bus\_83\_11$.
\label{fig:algo6}}
\end{figure}

\section{Conclusions}
\label{sec:conclusions}
We have considered three critical problems in radial distribution grids: learning the operational radial structure (Task $1$), inferring the nodal load statistics (Task $2$), and estimating the impedance parameters of operational lines (Task $3$). In Part I \cite{distgridpart1}, we have presented a polynomial time algorithm that uses nodal voltage magnitude samples, and available information on nodal injection statistics and line parameters to accomplish Task $1$. The algorithm is based on the assumption of second moments (of nodal power injections) positivity. In Part II, we have assumed independence of fluctuations in nodal injections instead and used it to develop a new polynomial time algorithm that solves Task $1$ coupled with either Task $2$ or Task $3$. Importantly, under our modified assumption, voltage magnitude measurements appear sufficient to learn the operational radial grid, even in the absence of any information on the line parameters or injection statistics. Availability  of the additional voltage phasor measurements is required to complete Tasks $2$ and $3$. Then, we have presented the second algorithm to estimate the grid structure for systems with incomplete observability, where voltage magnitude measurements for a set of missing nodes are not available. It is worth mentioning that the assumptions in Parts I and II of this paper, though different, simultaneously hold true for several realistic grids and time-scales. Moreover, neither assumption relies on any particular distribution for nodal injections. Performance of both Algorithms have been elucidated through simulations of a number of grid test cases. Apart from using these results to detect failures and also to improve load control, this work has key implications in related areas of non-intrusive control and quantification of measurement security and prevention of adversarial attacks. Learning the grid structure under generalized power flow models and related error analysis remain two interesting directions for future work.

\section*{Acknowledgment}

The work at LANL was funded by the Advanced Grid Modeling Program in the Office of Electricity in the US Department of Energy and was carried out under the auspices of the National Nuclear Security Administration of the U.S. Department of Energy at Los Alamos National Laboratory under Contract No. DE-AC52-06NA25396.

\bibliographystyle{IEEETran}
\bibliography{../../Bib/FIDVR,../../Bib/SmartGrid,../../Bib/voltage,../../Bib/trees}

\begin{IEEEbiography}[{\includegraphics[width=1in,height=2in,clip,keepaspectratio]{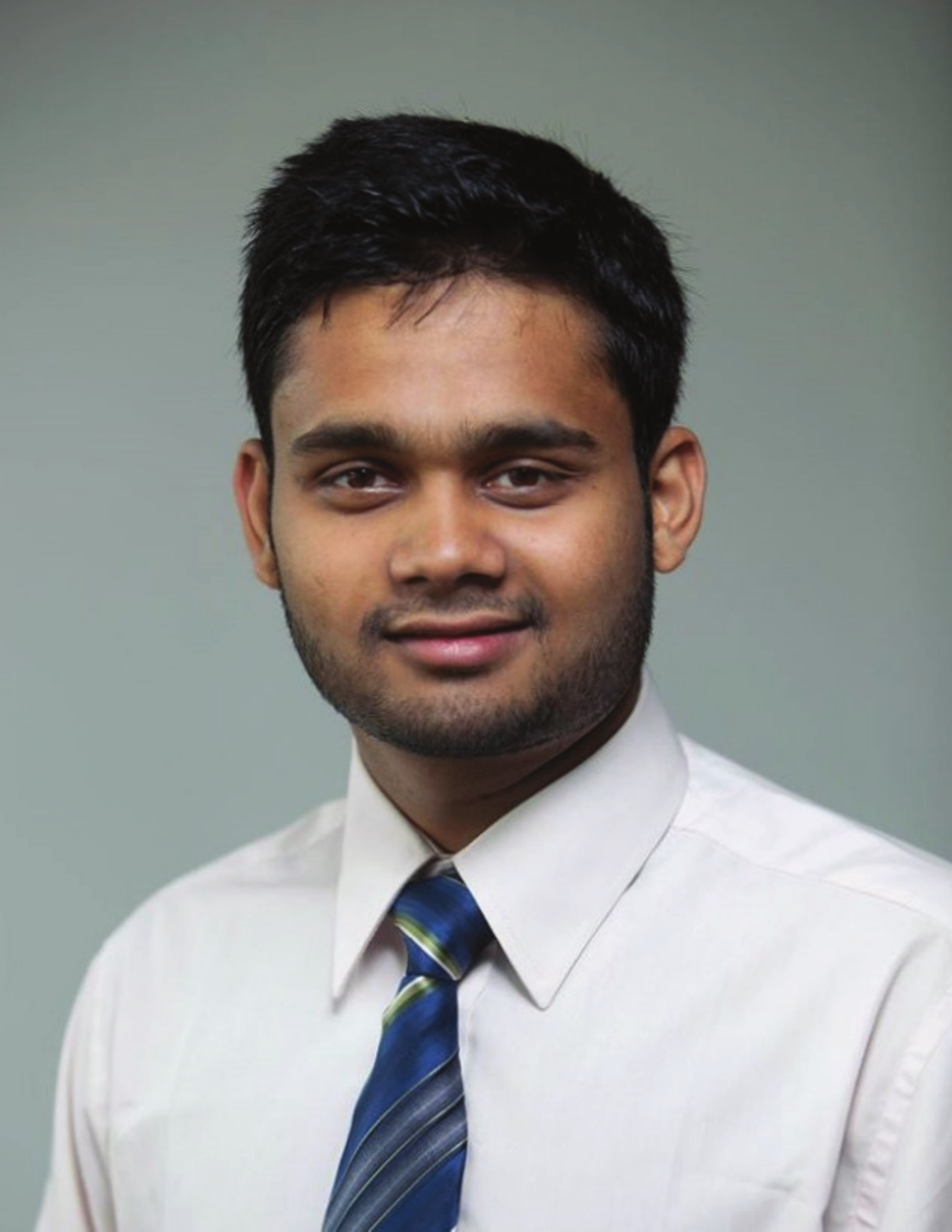}}]{Deepjyoti Deka}
Deepjyoti Deka received his M.S. in Electrical Engineering from University of Texas, Austin in 2011, and his B.Tech in Electronics and Communication Engineering from IIT Guwahati, India, in 2009 for which he was awarded the Institute Silver Medal. He is currently a PhD candidate in Electrical Engineering at UT Austin. His research focusses on the design and analysis of power grid structure, operations and data security. He is also interested in modeling and optimization in social and physical networks. He has held internship positions at Los Alamos National Lab, Los Alamos NM, Electric Reliability Council of Texas, Taylor TX, and Qualcomm Inc, San Diego CA.
\end{IEEEbiography}

\begin{IEEEbiography}[{\includegraphics[width=1in,height=2in,clip,keepaspectratio]{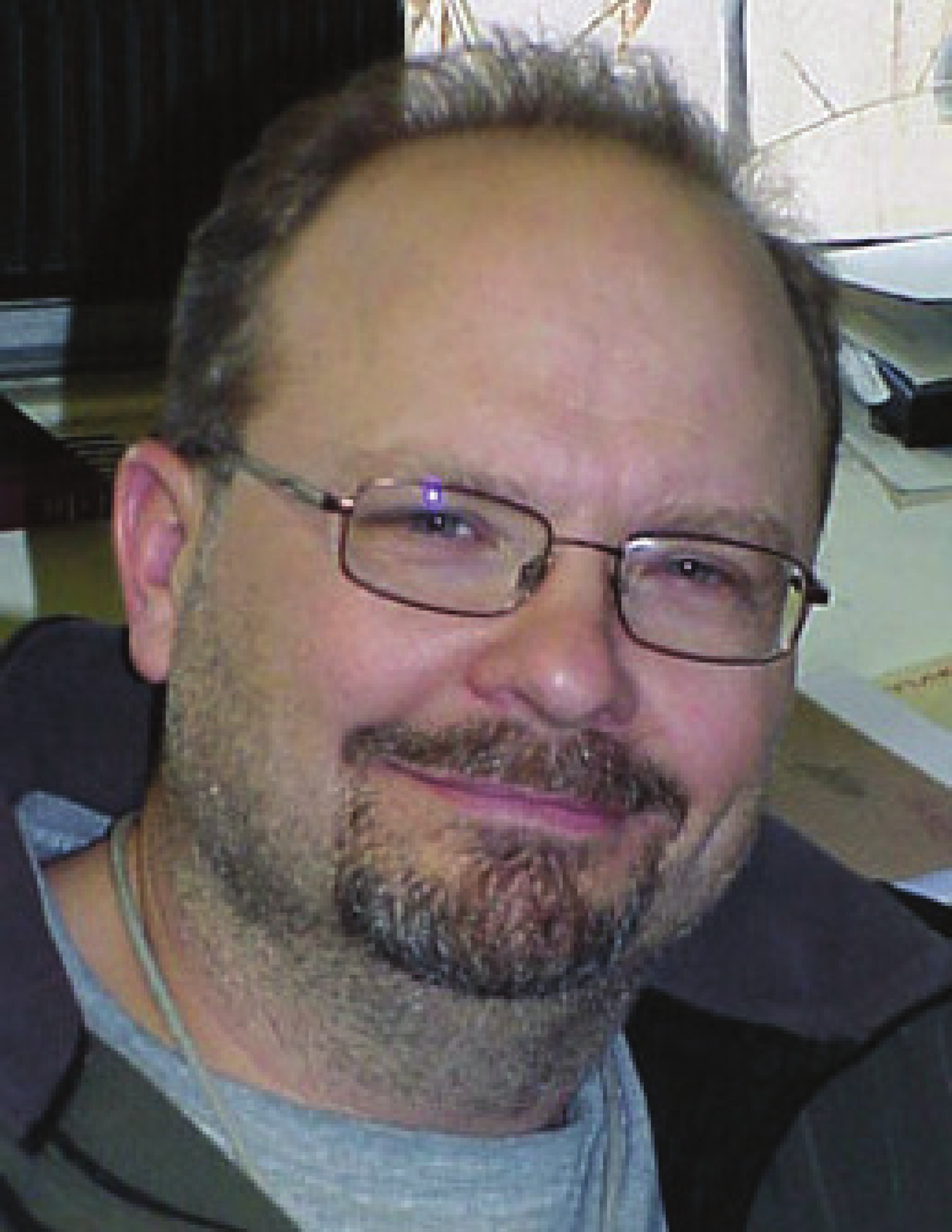}}]{ScottBackhaus}
Scott Backhaus received the Ph.D. degree in physics from the University
of California at Berkeley in 1997 in the area of experimental macroscopic
quantum behavior of superfluid He-3 and He-4.
In 1998, he came to Los Alamos, NM, was Director's Funded
Postdoctoral Researcher from 1998 to 2000, a Reines Postdoctoral
Fellow from 2001 to 2003, and a Technical Staff Member from 2003 to
the present. While at Los Alamos, he has performed both experimental
and theoretical research in the area of thermoacoustic energy conversion
for which he received an R\&D 100 award in 1999 and Technology
Review's Top 100 Innovators Under 35 [award in 2003]. Recently, his
attention has shifted to other energy-related topics including the fundamental
science of geologic carbon sequestration and grid-integration of
renewable generation.
\end{IEEEbiography}

\begin{IEEEbiography}[{\includegraphics[width=1in,height=1.25in,clip,keepaspectratio]{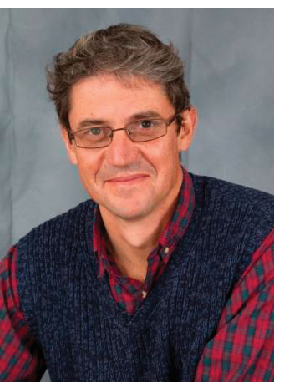}}]{Michael Chertkov}
Dr. Chertkov's areas of interest include statistical and
mathematical physics applied to energy and communication networks,
machine learning, control theory, information theory, computer science,
fluid mechanics and optics. Dr. Chertkov received his Ph.D. in
physics from the Weizmann Institute of Science in 1996, and his
M.Sc. in physics from Novosibirsk State University in 1990.
After his Ph.D., Dr. Chertkov spent three years at Princeton
University as a R.H. Dicke Fellow in the Department of Physics.
He joined Los Alamos National Lab in 1999, initially as a J.R.
Oppenheimer Fellow in the Theoretical Division. He is now a
technical staff member in the same division. Dr. Chertkov has
published more than 130 papers in these research
areas. He is an editor of the Journal
of Statistical Mechanics (JSTAT), associate editor of IEEE Transactions on
Control of Network Systems, a fellow of the American Physical
Society (APS), and a Founding Faculty Fellow of Skoltech (Moscow, Russia).
\end{IEEEbiography}

\end{document}